\documentclass[11pt,a4paper,leqno]{amsart}

\usepackage[latin1]{inputenc}
\usepackage[T1]{fontenc}
\usepackage{amsfonts}
\usepackage{amsmath}
\usepackage{amssymb}
\usepackage{eurosym}
\usepackage{mathrsfs}
\usepackage{palatino}
\usepackage{color}
\usepackage{esint}
\usepackage{url}
\usepackage{verbatim}
\usepackage{enumerate}
\usepackage[dvipsnames]{xcolor}

\usepackage{caption,float}
\usepackage{placeins}
\usepackage{afterpage}

\usepackage{subfig}

\usepackage[pagebackref,hypertexnames=false, colorlinks, citecolor=RoyalPurple, linkcolor=RoyalPurple, urlcolor=RoyalPurple]{hyperref}
\usepackage{amsrefs}
\newfloat{fig}{tb}{lop}[section]
\DeclareCaptionLabelFormat{Ucase}{Figure~#2}
\usepackage{graphicx}
\graphicspath{ {./pics_pbc/} }

\newcommand{\R}{\mathbb{R}}
\newcommand{\C}{\mathbb{C}}

\newcommand{\N}{\mathbb{N}}

\newcommand{\calX}{\mathcal{X}}

\newcommand{\calR}{\mathcal{R}}

\newcommand{\diam}{\operatorname{diam}}
\newcommand{\calJ}{\mathcal{J}}

\newcommand{\vare}{\varepsilon}

\numberwithin{equation}{section}

\newcommand{\ud}[0]{\,\mathrm{d}}

\newcommand{\esssup}[0]{\operatornamewithlimits{ess\,sup}}

\newcommand{\abs}[1]{|#1|}
\newcommand{\babs}[1]{\big|#1\big|}
\newcommand{\Babs}[1]{\Big|#1\Big|}

\newcommand{\Norm}[2]{\|#1\|_{#2}}
\newcommand{\bNorm}[2]{\big\|#1\big\|_{#2}}
\newcommand{\BNorm}[2]{\Big\|#1\Big\|_{#2}}

\newcommand{\ave}[1]{\langle #1\rangle}

\newcommand{\BMO}[0]{\operatorname{BMO}}
\newcommand{\bmo}[0]{\operatorname{bmo}}
\newcommand{\Dil}[0]{\operatorname{Dil}}

\newcommand{\loc}[0]{\operatorname{loc}}

\newcommand{\sign}[0]{\operatorname{sgn}}

\newcommand{\calD}[0]{\mathcal{D}}

\newcommand{\wt}[1]{{\widetilde{#1}}}

\swapnumbers
\theoremstyle{plain}
\newtheorem{thm}[equation]{Theorem}
\newtheorem{lem}[equation]{Lemma}
\newtheorem{prop}[equation]{Proposition}

\theoremstyle{definition}
\newtheorem{defn}[equation]{Definition}

\theoremstyle{remark}

\newtheorem{prbl}[equation]{Problem}
\newtheorem{rem}[equation]{Remark}

\title{Lower bound of the parabolic Hilbert commutator}

\author{Tuomas Oikari}
\address[T.O.]{Department of Mathematics and Statistics, University of Helsinki, P.O.B. 68, FI-00014 University of Helsinki, Finland}
\email{tuomas.v.oikari@helsinki.fi}

\makeatletter
\@namedef{subjclassname@2020}{%
	\textup{2010} Mathematics Subject Classification}
\makeatother

\thanks{T. Oikari was supported by the Academy of Finland project numbers 306901 and 314829, by the Finnish Centre of Excellence in Analysis and Dynamics Research project No. 307333, by the three-year research grant of the University of Helsinki No. 75160010 and by the Jenny and Antti Wihuri Foundation.
}

\begin{document}

\begin{abstract} Answering a key point left open in the recent work of  Bongers, Guo, Li and Wick \cite{Bongers2019commutators}, we provide the lower bound
	\[
	\Norm{b}{\BMO_{\gamma}(\R^2)}\lesssim \Norm{[b,H_{\gamma}]}{L^p(\R^2)\to L^p(\R^2)},
	\]
	where $H_{\gamma}$ is the parabolic Hilbert transform.
\end{abstract}
	\maketitle

\section{Introduction}
The commutator of the parabolic Hilbert transform,
\begin{align*}
	[b,H_{\gamma}]f(x) = b(x)H_{\gamma}f(x)-H_{\gamma}(bf)(x),\qquad H_{\gamma}f(x) = p.v.\int_{\R} f(x-\gamma(t))\frac{\ud t}{t}, 
\end{align*}
where $b\in L^1_{\loc}(\R^2;\C),$ $\gamma(t) = (t,t^2)$ and $f:\R^2\to\C,$
was recently studied in Bongers et al. \cite{Bongers2019commutators}, where they prove the following commutator estimates
\begin{align}\label{upper}
\Norm{b}{test}\lesssim \Norm{[b,H_{\gamma}]}{L^p(\R^2)\to L^p(\R^2)} \lesssim 	\Norm{b}{\BMO_{\gamma}(\R^2)}.
\end{align}
The upper bound involves the parabolic bmo norm
\begin{align*}
\Norm{b}{\BMO_{\gamma}(\R^2)}= \sup_{Q\in\calR_{\gamma}}\fint_{Q}\abs{b-\ave{b}_Q},
\end{align*}
where $\calR_{\gamma}$ is the collection of parabolic rectangles, i.e., rectangles $R=I\times J$ in the plane parallel with the coordinate axes such that $\ell(J) = \ell(I)^2.$
The lower bound however involves the non-matching testing condition
\begin{align*}
\Norm{b}{test} = \sup_{Q\in\calR_{\gamma}}\fint_Q\babs{b(x)- \frac{1}{\mu\big(I_{x,E_Q}\big)}\int_{I_{x,E_Q}}b(x-\gamma(t))\ud\mu(t)}\ud x,
\end{align*}
where $\mu(t) = \frac{\ud t}{t}$ and
\begin{align*}
E_Q = \big\{x-\gamma(t): x\in Q, t\in [9\ell(I),10\ell(I)] \big\},\qquad I_{x,E_Q} = \big\{t\in \R: x-\gamma(t)\in E_Q\big\}.
\end{align*}
Often, the necessity (the lower bound) is even more challenging than the corresponding sufficiency (the upper bound). In \cite{Bongers2019commutators} the necessity was left open and we provide a proof here, thus completing the picture.  Our main result is the following Theorem \ref{thm:main1}.
\begin{thm}\label{thm:main1} Let $b\in L^1_{\loc}(\R^2;\C)$ and $ p\in(1,\infty).$ Then, 
	\begin{align*}
	\Norm{b}{\BMO_{\gamma}(\R^2)} \lesssim\Norm{[b,H_{\gamma}]}{L^p(\R^2)\to L^p(\R^2)}.
	\end{align*}
\end{thm}
Taken together the lower bound in Theorem \ref{thm:main1} and the upper bound in \eqref{upper} allow us to conclude the following.
\begin{thm}\label{thm:main1B} Let $b\in L^1_{\loc}(\R^2;\C)$ and $ p\in(1,\infty).$ Then, 
	\begin{align*}
	\Norm{b}{\BMO_{\gamma}(\R^2)} \sim \Norm{[b,H_{\gamma}]}{L^p(\R^2)\to L^p(\R^2)}.
	\end{align*}
\end{thm}
We prove Theorem \ref{thm:main1} with a non-trivial adaptation of the approximate weak factorization argument.

\vspace{0.2cm}

The approximate weak factorization (awf) argument for proving commutator lower bounds for singular integral operators (SIOs) was recently developed and applied in Hyt\"{o}nen \cite{HYT2021JMPA} to complete the following picture.
 Let $1<p,q<\infty,$ $b\in L^1_{\loc}(\R^d)$ and $T$ be a non-degenerate Calderón-Zygmund operator (CZO), then
\begin{align}\label{line1}
\Norm{[b,T]}{L^p(\R^d)\to L^q(\R^d)} \sim	
\begin{cases}
\Norm{b}{\BMO(\R^d)}, & q = p,\quad \text{\cite{CRW}}\quad (1976), \\
\Norm{b}{\dot C^{\alpha,0}(\R^d)},\quad \alpha = d\big(\frac{1}{p}-\frac{1}{q}\big), & q>p, \quad \text{\cite{Janson1978}}\quad (1978), \\
\Norm{b}{\dot{L}^s(\R^d)},\quad \frac{1}{q} = \frac{1}{s}+\frac{1}{p},  & q<p,\quad \text{\cite{HYT2021JMPA}}\quad (2018).
\end{cases}
\end{align}
The commutator in \eqref{line1} is defined by $[b,T]f = bTf-T(bf),$
and the listed references are Coifman, Rochberg, Weiss \cite{CRW} and Janson \cite{Janson1978}.
The awf argument is strong in that it gives a unified approach to all of the three cases, in that it works for many singular integrals with kernels satisfying only minimum non-degeneracy assumptions, and in that it is flexible enough to grant e.g. multi-parameter and multilinear extensions. 
For the multi-parameter variants of the awf argument see Airta, Hyt\"{o}nen, Li, Martikainen, Oikari \cite{AHLMO2021} and Oikari \cite{Oik2020(1)},
where, respectfully, the commutators
\begin{align}\label{line2}
\big[T_2,[T_1,b]\big],\big[b,T\big]:L^{p_1}(\R^{d_1};L^{p_2}(\R^{d_2}))\to L^{q_1}(\R^{d_1};L^{q_2}(\R^{d_2}))
\end{align}
were treated. On the line \eqref{line2}, $1<p_1,p_2,q_1,q_2<\infty,$ $T_i$ is a one-parameter CZO on $\R^{d_i},$ for $i=1,2,$ and $T$ is a bi-parameter CZO on $\R^{d_1+d_2}.$ 
The adaptation of the awf argument to the bi-parameter settings was not effortless and for both commutators on the line \eqref{line2} the characterization of some cases is still open. 
For the multilinear extension see Oikari \cite{Oik2021(1)}. 

Another often-used argument, next to the awf argument, for proving commutator lower bounds is through the median method. The median method can only handle real-valued functions $b,$ however, the advantage is that it works for iterated commutators. For an account of the median method see \cite{HYT2021JMPA}. 

Commutators have of course been studied outside the aforementioned research articles and for some additional historically significant developments, we direct the reader to Nehari \cite{Nehari1957} (the case $q=p$ with the Hilbert transform) and Uchiyama \cite{Uch1978} (compactness of commutators and the case $q=p$ with any Riesz transform). Lastly, we mention the notable recent developments of Lerner, Ombrosi, Rivera-Ríos \cite{LOR2} and Guo, Lian, Wu \cite{guo2017unified}, before \cite{HYT2021JMPA}, that both recognized good non-degeneracy assumptions for commutator lower bounds. 


Commutator estimates, for example, imply factorization results for Hardy spaces, see \cite{CRW}, they have applications in PDEs through compensated compactness and div-curl lemmas, and they have played a major role in investigations of the Jacobian problem, see Coifman, Lions, Meyer and Semmes \cite{CLMS1993}, Lindberg \cite{Lindberg2017}, and \cite{HYT2021JMPA}. It is crucial in these applications that we have both commutator upper and lower bounds. 
%

%

In this article we almost solely focus on commutator lower bounds. For the convenience of the reader, we recall some of the timely developments in the theory of commutator upper bounds. The rough rule of thumb is that the upper bounds in the cases $q\not= p$ are easy and the main work lies with the case $q=p.$ Concerning the case $q=p,$ a modern sparse domination proof of the linear (also, essentially the multilinear) case can be found in \cite{LOR2}; for a proof of the multi-parameter cases through dyadic decomposition techniques we refer the reader to Holmes, Petermichl and Wick \cite{HPW}, and to Li, Martikainen and Vuorinen \cite{LMV2019prodbmocom}, \cite{LMV2019biparcom}.

%

%

\vspace{0.2cm}
The settings considered in all the aforementioned research articles, apart from \cite{AHLMO2021} which treats an iterated commutator of product nature, are such that the dimension of the ambient space $\R^d$ is the same as that of the singular integral.
Detaching from this, we consider singular integrals in the plane that are lower dimensional compared to the functions they hit, i.e., the kernel is localized to a curve.  The challenge in adapting the awf argument to the parabolic setting lies with the fact that a priori a curve can only record one dimensional information, whereas the parabolic bmo involves a truly two dimensional quantity. 
This mismatch brings new elements to the awf argument and necessitates a construction of a new kind of geometry compared to those present in the previous cases.

The main idea of the proof of Theorem \ref{thm:main1} can however be recorded in a model situation that involves only lines, in contrast to curves.
Let us recall the  directional Hilbert transforms:
\begin{align*}
H_{\sigma}f(x) = p.v.\int_{\R}f(x-\sigma t)\frac{\ud t}{t},\qquad \sigma\in\mathbb{S}^1,\qquad f:\R^2\to\C.
\end{align*}
Let $\calR$ denote the collection of all rectangles in the plane parallel to the coordinate axes. Then, the little bmo space is defined by the norm
\begin{align*}
	\Norm{b}{\bmo(\R^2)}= \sup_{R\in\calR}\fint_{R}\abs{b-\ave{b}_R}.
\end{align*}
Our second result is the following.
\begin{thm}\label{thm:main2} Let $b\in L^1_{\loc}(\R^2;\C)$ and $p\in(1,\infty).$ Then,
	\begin{align*}
	\Norm{b}{\bmo(\R^2)} \sim \sum_{i=1,2}\Norm{[b,H_{e_i}]}{L^p(\R^2)\to L^p(\R^2)},
	\end{align*}
	where $e_1 = (1,0)$ and $e_2 = (0,1).$
\end{thm}
Even though theorems \ref{thm:main1} and \ref{thm:main2} are independent, we recommend that the proof of Theorem \ref{thm:main2} is read first. 
As it is perhaps not clear that Theorem \ref{thm:main2} is non-trivial, next, as a reminder, we record the bi-parameter result that follows immediately by applying known results.
\begin{prop}\label{prop:trivial} Let $b\in L^1_{\loc}(\R^2;\C)$ and $p\in(1,\infty).$ Then,
	\begin{align*}
	\Norm{b}{\bmo(\R^2)} \sim \max\big(\esssup_{x_1\in\R}&\bNorm{[b(x_1,\cdot),H]}{L^p(\R)\to L^p(\R)}, \esssup_{x_2\in\R}\bNorm{[b(\cdot,x_2),H]}{L^p(\R)\to L^p(\R)}\big).
	\end{align*}
\end{prop}
\begin{proof} Follows by Lemma \ref{lem:bmo} (see below) and the one-parameter result $p=q$ recorded on the line \eqref{line1}.
\end{proof}
The following Lemma \ref{lem:bmo} was recorded at least in \cite{HPW}.
\begin{lem}\label{lem:bmo} Let $b\in L^1_{\loc}(\R^d;\C).$ Then,
	\begin{align*}
	\Norm{b}{\bmo(\R^2)} \sim \max\big(\esssup_{x_1\in\R}\Norm{b(x_1,\cdot)}{\BMO(\R)},\esssup_{x_2\in\R}\Norm{b(\cdot,x_2)}{\BMO(\R)}\big).
	\end{align*}
\end{lem}

\subsection{Basic notation}
We denote $L^1_{\loc}(\R^d;\C) = L^1_{\loc},$ $\int_{\R^d} = \int,$ and so on, mostly leaving out the ambient space if this information is obvious.

We denote averages with
$
\langle f \rangle_A = \fint_A f= \frac{1}{|A|} \int_A f,
$
where $|A|$ denotes the Lebesgue measure of the set $A$. The indicator function of a set $A$ is denoted by $1_A$.

We denote $z+ A  = \{z+a: a\in A\},$ $\Dil_{\lambda} A = \{z a: a\in A\}$ for $A\subset \R^2$ and $\lambda\in\R.$ 
For an interval $I\subset \R$ the centre-point is denoted $c_I$ and the concentric dilation is $\lambda I = [c_I-\lambda\frac{\ell(I)}{2},c_I+\lambda\frac{\ell(I)}{2}],$ for $\lambda>0.$  We also denote $-I = \Dil_{-1}I.$ 

A curve is a differentiable mapping $\gamma = (\gamma_1,\dots,\gamma_d):I\to\R^d$ parametrized over some interval $I.$  We denote curve length with $\ell(\gamma).$ 

We denote $A \lesssim B$, if $A \leq C B$ for some constant $C>0$ depending only on the dimension of the underlying space, on integration exponents and on other absolute constants appearing in the assumptions that we do not care about.
Then  $A \sim B$, if $A \lesssim B$ and $B \lesssim  A.$ Subscripts on constants ($C_{a,b,c,...}$) and quantifiers ($\lesssim_{a,b,c,...}$) signify their dependence on those subscripts. 

A large parameter $A>>1$ will appear throughout the text and tracking it is essential. Then, we write $\lesssim_A$ if and only if the said estimate depends on the parameter $A$, and if we write $X\lesssim C_A\cdot Y,$ then the implicit constant will never depend on the parameter $A.$

\subsection{Acknowledgements} I thank Emil Vuorinen for reading through the manuscript and for comments that led to improvements.

\section{Proof of Theorem \ref{thm:main2}}
\begin{proof}[Proof of Theorem \ref{thm:main2}] We first show that the commutator norms are bounded above by $\Norm{b}{\bmo}.$ This follows immediately by the standard boundedness theory of commutators and lemma \ref{lem:bmo},
	\begin{align*}
	\Norm{[b,H_{e_1}]f}{L^p(\R^2)} &= 	\BNorm{\bNorm{[b(x_1,x_2),H_{e_1}]f(x_1,x_2)}{L^p_{x_1}(\R)}}{L^p_{x_2}(\R)} \\ 
	&\leq \BNorm{\bNorm{[b(x_1,x_2),H]}{L^p_{x_1}(\R)\to L^p_{x_1}(\R)}\bNorm{f(x_1,x_2)}{L^p_{x_1}(\R)}}{L^p_{x_2}(\R)} \\
	&\lesssim \BNorm{\bNorm{b(x_1,x_2)}{\BMO_{x_1}(\R)}\bNorm{f(x_1,x_2)}{L^p_{x_1}(\R)}}{L^p_{x_2}(\R)} \\
	&\lesssim \esssup_{x_2\in\R}\Norm{b(x_1,x_2)}{\BMO_{x_1}(\R)}\bNorm{\bNorm{f(x_1,x_2)}{L^p_{x_1}(\R)}}{L^p_{x_2}(\R)} \\
	&\lesssim 	\Norm{b}{\bmo(\R^2)}\Norm{f}{L^p(\R^2)}.
	\end{align*}
	The other commutator norms are estimated similarly. We turn to the lower bound.
	
	Fix a rectangle $R_{0}= I\times J$ and a constant $A>1$ and define the three rectangles
	\begin{align*}
	R_1 = R_0+A\ell(I)e_1,\qquad R_2 = R_1 + A\ell(J)e_2,\qquad R_3 = R_2 - A\ell(I)e_1.
	\end{align*}
	Writing $\psi_{R_i}$ means that $\psi_{R_i}$ is a function supported on the set $R_i.$
	We begin with writing
	\begin{align}\label{eq:dualization}
	\int_{R_0}\abs{b-\ave{b}_{R_0}} = \int bf,\qquad f = (\theta-\ave{\theta}_{R_0})1_{R_0},\qquad \theta = \frac{\overline{b-\ave{b}_{R_0}}}{\abs{b-\ave{b}_{R_0}}}1_{\{b\not = \ave{b}_{R_0}\}},
	\end{align}
	and
		\begin{equation}\label{decomp1}
	\begin{split}
	f &=  \big[h_{R_0}H_{e_1}^*g_{R_1} - g_{R_1}H_{e_1}h_{R_0}\big] +  \wt{f}_{R_1} \\
	&=\big[h_{R_0}H_{e_1}^*g_{R_1} - g_{R_1}H_{e_1}h_{R_0}\big] + \big[h_{R_1}H_{e_2}^*g_{R_2}-g_{R_2}H_{e_2}h_{R_1}\big] + \wt{f}_{R_2},
	\end{split}
	\end{equation}
	where
	\begin{align*}
	h_{R_0} = \frac{f}{H_{e_1}^*g_{R_1}},\quad g_{R_i} = 1_{R_i},\quad \wt{f}_{R_1} = g_{R_1}H_{e_1}h_{R_0},\quad h_{R_1} = \frac{\wt{f}_{R_1}}{H_{e_2}^*g_{R_2}},\quad \wt{f}_{R_2} = g_{R_2}H_{e_2}h_{R_1}.
	\end{align*}
	The only possible problem in the above factorization of the function $f$ is a division by zero in $h_{R_i},$ $i=1,2,$ however, the estimates \eqref{wd} and \eqref{wdd} below show the denominators to be strictly positive functions.
	Next, we will show that
	\begin{align}\label{error}
	\abs{h_{R_0}}\lesssim_A 1_{R_0},\qquad \abs{h_{R_1}}\lesssim_A 1_{R_1},\qquad \Norm{ \wt{f}_{R_2}}{\infty} \lesssim A^{-1}\Norm{f}{\infty}.
	\end{align}
	We reserve the following notation for the variables: $x\in R_0,$ $y\in R_1,$ $z\in R_2$ and we denote
	\begin{align*}
	I(x,+) &= \{t\in\R: x+ te_1\in R_1\}, &I(y,-)&= \{t\in\R: y- te_1\in R_0\}, \\
	J(y,+)&= \{t\in\R: y+ te_2\in R_2\}, &J(z,-)&= \{t\in \R: z-e_2t\in R_1\}.
	\end{align*}
	Notice that $I(x,+),I(y,-)$ are intervals of length $\ell(I)$ containing the point $A\ell(I).$ Similarly, $J(y,+),	J(z,-)$ are intervals of length $\ell(J)$ containing the point $A\ell(J).$
	Fix a point $z\in R_2$ and write
	\begin{equation}\label{x}
	\begin{split}
		\wt{f}_{R_2}(z) &= H_{e_2}h_{R_1}(z) = \int_{J(z,-)}\frac{H_{e_1}h_{R_0}(z-e_2t)}{H_{e_2}^*g_{R_2}(z- e_2t)}\frac{\ud t}{t} \\
	&= \int_{J(z,-)}\frac{1}{H_{e_2}^*g_{R_2}(z- e_2t)}\int_{I(z-e_2t,-)}\frac{f(z-e_2t-e_1s)}{H_{e_1}^*g_{R_1}(z-e_2t-e_1s)}\frac{\ud s}{s}\frac{\ud t}{t} \\
	&= \int_{J(z,-)}\int_{I(z-e_2t,-)} \frac{ f(z-e_2t-e_1s)}{H_{e_2}^*g_{R_2}(z- e_2t)H_{e_1}^*g_{R_1}(z-e_2t-e_1s) }\frac{\ud s}{s}\frac{\ud t}{t}.
	\end{split}	
	\end{equation}
	Let $x\in R_0$ and $y\in R_1$ be arbitrary. Then, there holds that 
		\begin{align}\label{wdd}
	\frac{1}{A+1}\leq H_{e_1}^*g_{R_1}(x) = \int_{I(x,+)}\frac{\ud t}{t} \leq \frac{1}{A-1}
	\end{align}
	and
	\begin{align}\label{wd}
	\frac{1}{A+1}\leq H_{e_2}^*g_{R_2}(y) = \int_{J(y,+)}\frac{\ud t}{t} \leq \frac{1}{A-1}.
	\end{align}
	From \eqref{wd} and \eqref{wdd} it follows immediately that $\abs{h_{R_i}}\lesssim_A 1_{R_i},$ for $i=0,1,$ and hence for the claims on the line \eqref{error} it remains to check that  $\Norm{ \wt{f}_{R_2}}{\infty} \lesssim A^{-1}\Norm{f}{\infty}.$
	For arbitrary $t\in J(z,-)$ and $s\in I(z-e_2t,-),$ denoting
	\begin{align}\label{wddd}
		t' = A H_{e_2}^*g_{R_2}(z- e_2t)t,\qquad s' = AH_{e_1}^*g_{R_1}(z-e_2t-e_1s) s,
	\end{align}
	there holds that 
	\begin{align}\label{wdddd}
		\abs{t'-A\ell(J)}\lesssim \ell(J),\qquad \abs{s'-A\ell(I)}\lesssim \ell(I).
	\end{align}
	Let us briefly check the left estimate of \eqref{wdddd}.
	Assume e.g. that $A\ell(J)\leq t \leq (A+1)\ell(J),$ then by \eqref{wd} we find that
	\begin{align*}
			\abs{t'-A\ell(J)} &\leq \frac{A}{A-1}t-A\ell(J) \leq \frac{A+1}{A-1} A\ell(J)-A\ell(J) = \left(\frac{A+1}{A-1}A-A\right) \ell(J) \\
			&=		\frac{2A}{A-1}\ell(J) \lesssim \ell(J),
	\end{align*}
	whenever, say, $A>2.$ The other cases can be checked similarly.
	Now we come to the crucial part of the argument. The double integral in \eqref{x} is exactly over the rectangle $R_0,$ i.e., for all $z\in R_2$ there holds that
	\begin{align}\label{eq:crux1}
	R_0 = \{z-e_2t-e_1s:t\in J(z,-)\, , s\in I(z-e_2t,-)\}
	\end{align} 
	and hence that
	\begin{align}\label{aid5}
		\int_{J(z,-)}\int_{I(z-e_2t,-)}f(z-e_2t-e_1s)\ud s\ud t = \int_{R_0}f = 0.
	\end{align}
	By \eqref{aid5} we find that
	\begin{equation}\label{eq1}
	\begin{split}
	\eqref{x} =  A^2\int_{J(z,-)}\int_{I(z-e_2t,-)}f(z-e_2t-e_1s)\left(\frac{1}{t'\cdot s'}-\frac{1}{A\ell(J)\cdot A\ell(I)}\right)\ud s\ud t.
	\end{split}
	\end{equation}
	Now, applying the estimates on the line \eqref{wdddd}, that $t'\sim A\ell(J)$ and $s'\sim A\ell(I)$ (as $A$ is large, this is implied by \eqref{wdddd}), triangle inequality, and the mean value theorem (applied to $x\mapsto x^{-1}$ in the second passing), shows that
	\begin{equation*}
	\begin{split}
	\Babs{\frac{1}{t'\cdot s'}-\frac{1}{A\ell(J)\cdot A\ell(I)}} &\leq \frac{1}{t'}\Babs{\frac{1}{ s'}-\frac{1}{A\ell(I)}} + \frac{1}{ A\ell(I)}\Babs{\frac{1}{t'}-\frac{1}{A\ell(J)}} \\
	&\lesssim \frac{ (A\ell(I))^{-2}\ell(I)}{A\ell(J)} + \frac{(A\ell(J))^{-2}\ell(J)}{A\ell(I)}  \lesssim \frac{1}{A^3\ell(I)\ell(J)}.
	\end{split}
	\end{equation*}
	Plugging in the above estimates we continue from \eqref{eq1} and find that 
	\begin{align*}
	\abs{	\eqref{eq1}} \lesssim \Norm{f}{\infty}\int_{J(z,-)}\int_{I(z-e_2t,-)}\frac{1}{A\ell(I)\ell(J)}\ud s\ud t \leq A^{-1}\Norm{f}{\infty}
	\end{align*}
	and hence we have established \eqref{error}.
	
	Next, we repeat the above argument beginning from the function $\wt{f}_{R_2}.$ We denote $e_3 = -e_1$ and $e_0 = - e_2$ and write
	\begin{align}\label{decomp2}
	\wt{f}_{R_2} = \left[h_{R_2}H_{e_3}^*g_{R_3} - g_{R_3}H_{e_3}h_{R_2}\right] + \left[ h_{R_3}H_{e_0}^*g_{R_0}-g_{R_0}H_{e_0}h_{R_3}\right] + \wt{f}_{R_0},
	\end{align}
	where
	\begin{align*}
	h_{R_2} = \frac{\wt{f}_{R_2}}{H_{e_3}^*g_{R_3}},\quad g_{R_i} = 1_{R_i},\quad \wt{f}_{R_3} = g_{R_3}H_{e_3}h_{R_2},\quad h_{R_3} = \frac{\wt{f}_{R_3}}{H_{e_0}g_{R_0}},\quad \wt{f}_{R_0} = g_{R_0}H_{e_0}h_{R_3}.
	\end{align*}
	Again, this decomposition is well-defined.
	By moving the adjoints we find that
	\begin{align}\label{eq:zm}
	\int \wt{f}_{R_2} = 	\int \wt{f}_{R_1} = \int f = 0,
	\end{align}
	e.g. the second identity follows as
	\begin{align*}
	\int \wt{f}_{R_1} = \int g_{R_1}H_{e_1}\Big(\frac{f}{H_{e_1}^*g_{R_1}}\Big) = \int H_{e_1}^*g_{R_1}\frac{f}{H_{e_1}^*g_{R_1}} = \int f.
	\end{align*}
	Consequently, by similar arguments as above, we find that 
	\begin{align*}
	\abs{h_{R_i}}\lesssim_A 1_{R_i},\quad i=2,3,\qquad \Norm{\wt{f}_{R_0}}{\infty} \lesssim A^{-1}\Norm{\wt{f}_{R_2}}{\infty} \lesssim A^{-2}\Norm{f}{\infty}.
	\end{align*}
	
	Then, we dualize as on the line \eqref{eq:dualization} and factor according to the lines \eqref{decomp1} and \eqref{decomp2} to the extent that
 	\begin{align*}
	\int_{R_0}\abs{b-\ave{b}_{R_0}} &= \int b\sum_{i=1}^3\left[h_{R_{i-1}}H_{e_i}^*g_{R_i} - g_{R_i}H_{e_i}h_{R_{i-1}}\right] \\ 
		&\qquad\qquad+ \int b\left[ h_{R_3}H_{e_0}^*g_{R_0}-g_{R_0}H_{e_0}h_{R_3}\right]+ \int b\wt{f}_{R_0} \\
	&= -\int \sum_{i=1}^3  g_{R_i}[b,H_{e_i}]h_{R_{i-1}} \\
	&\qquad\qquad-\int g_{R_0}[b,H_{e_0}]h_{R_3} + \int (b-\ave{b}_{R_0})\wt{f}_{R_0} \\
		&\leq \sum_{i=1}^3\Norm{g_{R_i}}{L^{p'}}\bNorm{[b,H_{e_i}]}{L^p\to L^p} \Norm{h_{R_{i-1}}}{L^p} \\
	&\qquad\qquad + \Norm{g_{R_0}}{L^{p'}}\bNorm{[b,H_{e_0}]}{L^p\to L^p} \Norm{h_{R_3}}{L^p} + \Norm{\wt{f}_{R_0}}{\infty}\int_{R_0}\abs{b-\ave{b}_{R_0}} \\
	&\leq C_A\Big(\sum_{i=1}^3\bNorm{[b,H_{e_i}]}{L^p\to L^p} + \bNorm{[b,H_{e_0}]}{L^p\to L^p}  \Big)\abs{R_0} +  CA^{-1}\int_{R_0}\abs{b-\ave{b}_{R_0}} \\
	&\leq  C_A\sum_{i=1,2}\Norm{[b,H_{e_i}]}{L^p(\R^2)\to L^p(\R^2)}\abs{R_0} + CA^{-1}\int_{R_0}\abs{b-\ave{b}_{R_0}}, 
	\end{align*}
	 where in the final estimate we note that $H_{\sigma} = -H_{-\sigma},$ for any $\sigma\in\mathbb{S},$ especially then, $H_{e_3} = -H_{e_1},$ $H_{e_0} = -H_{e_2}$ so that
	$\Norm{[b,H_{e_3}]}{L^p\to L^p} = \Norm{[b,H_{e_1}]}{L^p\to L^p} $ and $\Norm{[b,H_{e_0}]}{L^p\to L^p} = \Norm{[b,H_{e_2}]}{L^p\to L^p}.$ To conclude, using the assumption $b\in L^1_{\loc},$ we choose $A$ large enough and absorb the common term shared on both sides to the left-hand side, then divide with $\abs{R_0}.$
\end{proof}

\section{Proof of Theorem \ref{thm:main1B}} 
Whereas Theorem \ref{thm:main2} was in a sense proved on the go, now, due to the parabola, the setup is more involved and we require a lengthier preparation. As the upper bound was already proved in \cite{Bongers2019commutators}, it remains to prove Theorem \ref{thm:main1}.

\subsection{Geometry behind the factorization}\label{ss:gf} 
\subsubsection{Setup for analysis: $Q,$ $W$ and $P$}
We fix a parabolic rectangle $Q = I\times J,$ i.e. $\ell(J) = \ell(I)^2.$ We work on a scale comparable to $\ell(I)$ and hence define the auxiliary interval
\begin{align*}
	 I_A = [\ell(I)A,\ell(I)(A+N)],\qquad A,N \geq 1.
\end{align*}
Then, we set
\begin{align*}
	P &= Q + (2A+N)\ell(I)e_1, \\
	\wt{Q} &= \big\{ x+\gamma(t): x\in Q,\, t\in I_A\big\},\qquad \wt{P} = \big\{z + \gamma(t): z\in P,\, t \in -I_A\big\}, \\ 
	W &= \wt{Q}\cap \wt{P}.
\end{align*}
The following Figure \ref{fig:all2} is a rough sketch of the sets $Q,\wt{Q},W,\wt{P},P,$ 
 when $A\sim 3, N\sim 7, \ell(I) \sim 2.$ 
\FloatBarrier
\begin{fig}[h]
	\centering
	\includegraphics[scale= 0.2]{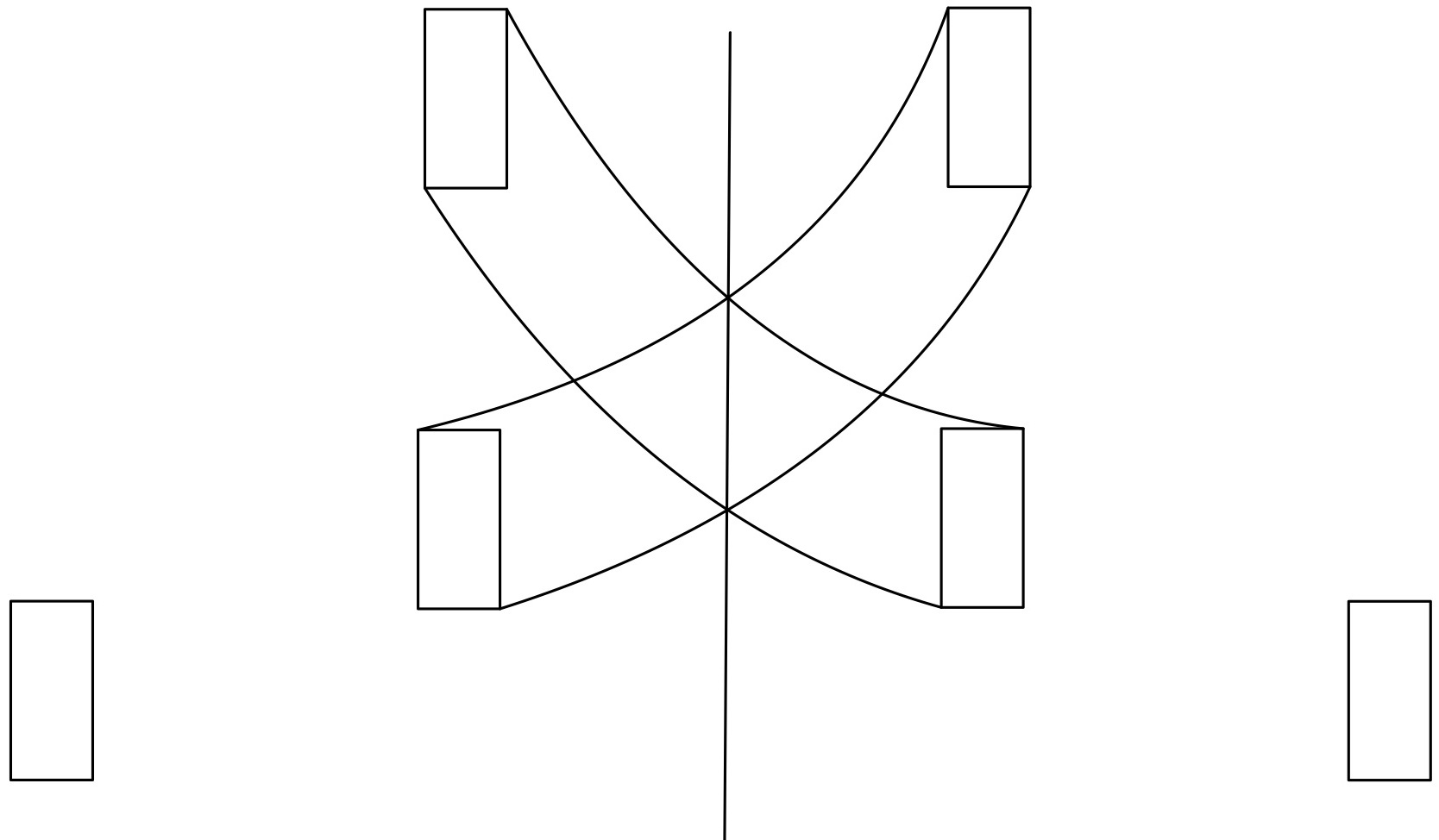}
	\caption{Setup for analysis}
	\label{fig:all2}
\end{fig}
For our arguments to work we can take any fixed $N\geq 1,$ and we take $N=1,$ however, considering a slightly larger $N$ brings separation to the sets considered and streamlines the geometry.

The setup is symmetric with respect to a reflection across the line in Figure \ref{fig:all2} that splits the set $W$ vertically in half.
Moreover, there holds that 
\begin{align}\label{size1}
\abs{Q}\sim_A \abs{\wt{Q}}  \sim  \abs{W} \sim  \abs{\wt{P}} \sim_A \abs{P}. 
\end{align}
The first estimate follows as $\wt{Q}$ contains a translate of $Q$ and by considering the size of the set $\wt{Q}$ in the $x_1$ and $x_2$ directions. The second follows as $W\subset\wt{Q}$ and $W$ contains a translate of $Q.$ That the second estimate is also independent of $A$ is a fact that we do not need, however, it is relatively clear from Lemma \ref{lem:geom1} below. The last two estimates are symmetric with the first two.

We denote
\begin{align}\label{denote}
	lb,lt,rb,rt,\qquad l = left,\quad r = right,\quad t= top,\quad b = bottom,\quad c = centre
\end{align}
and variables are reserved to be used as follows,
$
x\in Q,$ $y\in W,$ $z\in P.$
We also notate
\begin{align*}
I(a,\pm, B) = \{t\in\R : a \pm \gamma(t)\in B\}\subset \R,\qquad \phi(a,\pm,B) = \{a\pm\gamma(t)\in B: t\in \R\} \subset B.
\end{align*}
The variable $a\in\R^2$ is the reference point, the sign $\pm \in \{+,-\}$ is either the plus or the minus sign and indicates direction, while the last variable is a set $B\subset \R^2.$

\begin{lem}\label{lem:geom12}  Let $Q = I\times J \in\calR_{\gamma}$ and $x\in Q.$ Then, there holds that   
	\begin{align*}
	\lim_{A\to\infty}\frac{\abs{I(x,+,W)}}{\ell(I)} = \frac{1}{2}
	\end{align*}
	with uniform convergence independent of the data $x,Q.$
\end{lem}
\begin{proof} Let $x\in Q$ be arbitrary and let $t_x\sim A\ell(I)$ be the smallest number such that $x+\gamma(t_x)\in\partial W.$ 
	Let $-s_x\sim A\ell(I)$ be the point such that  $x+\gamma(t_x) = v_{lb}+\gamma(s_x),$
	where $v_{lb}$ is the left bottom vertex of $P.$
	Define the auxiliary point $y_x = x+\gamma(t_x)+\ell(I)e_1.$ Then, there holds that
	$y_x = v_{rb} +\gamma(s_x),$ where $v_{rb}$ is the right bottom vertex of $P.$
	
	Next, we show that for each $0<\vare<\frac{1}{2},$ there exists $A$ large enough (independent of $\ell(I)$) so that
	\begin{align}\label{auxe}
	\pi_2(x+\gamma(t_x+(\frac{1}{2}-\vare)\ell(I))) < \pi_2(v_{rb} + \gamma(s_x-\frac{\ell(I)}{2})).
	\end{align}
	To achieve \eqref{auxe}, we choose $A$ so large that 
	\begin{align}\label{auxcc}
	(\frac{1}{2}-\vare)\ell(I)\cdot 2(A+N)\ell(I) < \frac{1}{2}\ell(I)\cdot 2A\ell(I),
	\end{align}
	clearly $A$ as chosen on the line \eqref{auxcc} is independent of $\ell(I).$ Then, as 
	\[
	2A\ell(I)\leq\abs{\gamma_2'(t)}\leq 2(A+N)\ell(I),\qquad t\in -I_A\cup I_A
	\]
	and $\pi_2(x+\gamma(t_x)) = \pi_2(v_{rb}+\gamma(s_x)),$ \eqref{auxe} follows.
	By symmetry, 
	\begin{align}\label{auxd}
	\pi_2(x+\gamma(t_x+\frac{\ell(I)}{2}) > \pi_2(v_{rb} + \gamma(s_x-(\frac{1}{2}-\vare)\ell(I))).
	\end{align}
	
	Then, let $0<-u_x,v_x < \ell(I)$ be the unique points such that $x+\gamma(t_x+v_x) = v_{rb}+\gamma(s_x+u_x).$ The line \eqref{auxe} shows that $v_x \geq (\frac{1}{2}-\vare)\ell(I).$ 
	Indeed, assume for contradiction that $v_x < (\frac{1}{2}-\vare)\ell(I).$ Then by $v_x-u_x = \ell(I)$ necessarily $-u_x > s_x-\frac{\ell(I)}{2}$ and hence by \eqref{auxe}
	\[
	\pi_2\big(x+\gamma(t_x+v_x)\big) <  \pi_2\big(v_{rb} + \gamma(s_x-\frac{\ell(I)}{2})\big) < \pi_2\big(v_{rb} + \gamma(s_x+u_x)\big),
	\]
	which contradicts $x+\gamma(t_x+v_x) = v_{rb}+\gamma(s_x+u_x).$
	Similarly, from \eqref{auxd} it follows that $u_x \leq -(\frac{1}{2}-\vare)\ell(I)$. 
	Using $v_x - u_x = \ell(I),$ it follows that
	\begin{align}\label{aidd}
	-u_x,v_x\in [(\frac{1}{2}-\vare)\ell(I), (\frac{1}{2}+\vare)\ell(I)].
	\end{align} 
	Notice that  
	\begin{align}\label{odour}
	x+\gamma(t_x+v_x + \frac{\ell(I)}{2A})\not\in \wt{P}.
	\end{align}
	Indeed, \eqref{odour} follows from the information
		\[
		v_{rb}+\gamma(s_x+u_x) = x+\gamma(t_x+v_x),\qquad \abs{\gamma_2'(t)} > 2A\ell(I),
		\]
		which implies that
		\[
		x+\gamma(t_x+v_x+\frac{\ell(I)}{2A})\not\in \gamma(s_x+u_x+h) + P,\qquad h \geq 0,  
		\]
		along with the obvious fact that 
		\[
		x+\gamma(t_x+v_x+\frac{\ell(I)}{2A})\not\in \gamma(s_x+u_x+h) + P,\qquad h < 0.
		\]
	From \eqref{aidd} and \eqref{odour} we find that
	\begin{align}\label{auxr}
	(\frac{1}{2}-\vare)\ell(I) \leq  \abs{I(x,+,W)} \leq (\frac{1}{2}+\vare+\frac{1}{2A})\ell(I).
	\end{align}
	Clearly \eqref{auxr} implies the claim.
\end{proof}

Lemma \ref{lem:geom12} immediately gives as a corollary the following Lemma \ref{lem:geom1}.
\begin{lem}\label{lem:geom1}	
Let $x\in Q$ be arbitrary. Then,
\begin{align*}
	\abs{I(x,+,W)}\sim \ell(I).
\end{align*}
Also, there holds that 
\begin{align*}
\lim_{A\to\infty}\sup_{\substack{Q\in\calR_{\gamma} \\ x,x'\in Q }} \frac{\abs{I(x,+,W)}}{\abs{I(x',+,W)}} = 1.
\end{align*}
\end{lem}

Towards the next lemma define the reference rectangles 
$$
R(r) = [0,2^{-(r-1)}\frac{\ell(I)}{A}]\times [0, 2^{-r}\ell(I)^2],\qquad 1\leq r <\infty.
$$ Then, we set 
\begin{align*}
	P^{lb}(r) = v_{lb} + R(r),\qquad P^{rt}(r) = v_{rt} + \Dil_{-1}R(r),\qquad 	P^{lb}(r),P^{rt}(r)\subset P
\end{align*}
and define
\begin{equation}\label{K}
	\begin{split}
		\Delta_r(P^{lb}) &=  \overline{\partial P^{lb}(r)\setminus\partial P}, \qquad  
	\Delta_r(P^{rt})  =\overline{\partial P^{rt}(r)\setminus\partial P}, \\ 
	P^c &= P\setminus\bigcup_{r>1}\big(\Delta_r(P^{lb}) \cup 	\Delta_r(P^{rt})\big).
	\end{split}
\end{equation}

Notice that $P^c = P\setminus\left(P^{lb}(1)\cup P^{rt}(1)\right).$
If $z\in P^c$ ($c$ for centre) and $y\in\phi(z,+,W),$ then $\babs{I(y,-,P)}\sim \frac{\ell(I)}{A}$ (relatively clear, also, see Lemma \ref{lem:geom2} below). The following Lemma \ref{lem:geom2} shows that the sets $	\Delta_r(P^{lb}),	\Delta_r(P^{rt})$ exactly quantify this same statement for points situated towards the vertices $v_{lb}, v_{rt}$ of $P.$

\begin{lem}\label{lem:geom2} 
	Let $z\in \Delta_r(P^{lb})\cup \Delta_r(P^{rt})$ and $y\in\phi(z,+,W).$ Then, 
	\begin{align*}
	\abs{I(y,-,P)}&\sim 2^{-r}\frac{\ell(I)}{A}.
	\end{align*}
	Let $z\in P^c$ and $y\in\phi(z,+,W),$ then $	\abs{I(y,-,P)}\sim \frac{\ell(I)}{A}.$
\end{lem}
\begin{proof}
	Let $z\in \Delta_r(P^{lb})$ for some $r> 1.$ Then, either
	$$
	\pi_1(z-v_{lb}) = 2^{-(r-1)}\frac{\ell(I)}{A}\quad \mbox{ or }\quad  \pi_2(z-v_{lb}) = 2^{-r}\ell(I)^2
	$$
	holds.
	Assume first that $\pi_1(z-v_{lb}) = 2^{-(r-1)}\frac{\ell(I)}{A}.$
	As  $y\in\phi(z,+,W),$ there exists $s\sim -A\ell(I)$ so that $y-\gamma(s) = z\in P.$ The claim will follow if we show the following: there exists an absolute constant $c>0$ so that
	\begin{align}\label{ic}
			\pi_i(y-\gamma(s+h)) \in \pi_i(P),\qquad h\in (0,c2^{-(r-1)}\frac{\ell(I)}{A}),\qquad i=1,2.
	\end{align}
	The case $i=1$ is an immediate consequence of the following information
	\begin{align*}
		\pi_1(y-\gamma(s+h)) = z_1-h,\qquad z_1-\pi_1(v_{lb}) = 	2^{-(r-1)}\frac{\ell(I)}{A}\qquad h\in (0,c2^{-(r-1)}\frac{\ell(I)}{A}),
	\end{align*}
	as long as we choose $c$ small enough.  For the case $i=2$, we note that $(s,s+h)\subset -2I_A$ is an interval of length $h$ and hence 
	for some absolute constant $c_1>0$ there holds that
	\begin{align}\label{aa}
		\ell(\pi_2(\gamma(s,s+h))) \leq c_1hA\ell(I) \leq c_1c2^{-(r-1)}\frac{\ell(I)}{A}A\ell(I) \leq \frac{1}{2}\ell(I)^2,
	\end{align}
	as long as we choose $c$ small enough.
	The inequality \eqref{aa} implies that $\pi_2(y-\gamma(s+h)) \leq \pi_2(v_{lt}).$ Also clearly  $\pi_2(y-\gamma(s+h)) \geq \pi_2(v_{lb}).$ Together these show that $\pi_2(y-\gamma(s+h)) \in \pi_2(P)$  and so we have also checked the case $i=2$ on the line \eqref{ic}. 
	
	The case $ \pi_2(z-v_{lb}) = 2^{-r}\ell(I)^2$ and then $z\in \Delta_r(P^{rt})$ are handled in very much the same way and we leave the details to the reader.
\end{proof}

\subsubsection{Auxiliary functions}\label{sect:auxf}
Recall, that a fixed parabolic rectangle $Q\in\calR_{\gamma}$ and the parameters $A,N$ determine the sets $W (=W(Q)),$ $P ( = P(Q)).$  During the factorization we will  in total make use of four auxiliary functions, the first two are particularly simple, $g_Q = 1_Q,$ $g_P = 1_P.$ 
The other two functions are supported on the set $W,$ more precisely, we find two collections of functions, the first one being $\{g_W\}_{Q\in\calR_{\gamma}},$ and we show that the following three conditions are met.
\begin{enumerate}[$(i)$]
	\item There holds that
	\begin{align}\label{q0}
	1_{W^c}g_W = 1_{W^c},\qquad W^c= \big\{y\in W: \exists z\in P^c: y\in \phi(z,+,W)\big\}.
	\end{align}
	\item There holds that
	\begin{align}\label{q00}
	 g_W(y)\sim\abs{I(y,-,P)}\frac{A}{\ell(I)},
	\end{align}
	where the implicit constants do not depend on $y,Q.$ 
	\item  There holds that 
	\begin{align}\label{q1}
		\lim_{A\to\infty}\sup_{\substack{ Q\in\calR_{\gamma} \\ z\in P \\  t,t'\in I(z,+,W) }}\frac{g_W(z+\gamma(t))}{g_W(z+\gamma(t'))} = 	 1.
	\end{align}
\end{enumerate}
The reader who feels comfortable with the existence of such a family $\{g_W\}_{Q\in\calR_{\gamma}}$ may immediately skip to Section \ref{sect:awf}. 

Next we explicitly define the functions $g_W.$ 
Let $\eta_Q\geq 0$ be the smallest constant so that the following is a partition,
\begin{align*}
	W = \bigcup_{s\in [-\eta_Q,\infty)}W(s),\qquad W(s) = \Big\{y\in W: 	\abs{I(y,-,P)} = 2^{-s}\frac{\ell(I)}{A} \Big\}.
\end{align*}
Lemma \ref{lem:geom2} implies that $\sup_{Q\in\calR_{\gamma}}\eta_Q<\infty,$ a fact worth noting, which, however, we do not need anywhere.
Then, we define
\begin{align*}
	\varphi:W\times\R_+\to\R_+,\qquad \varphi(y,M) = 
	\sum_{-\eta_Q\leq s < M} 1_{W(s)}(y) +  \sum_{s\geq M}1_{W(s)}(y)2^{-(s-M)}.
\end{align*}
By the following Lemma \ref{auxfun} we choose $M$ large enough and define 
\begin{align*}
	g_W(\cdot) = \varphi(\cdot,M).
\end{align*}
\begin{lem}\label{auxfun} There exists $M\in\R_+$ so that $\varphi(\cdot,M)$ satisfies the points $(i),(ii)$ and $(iii).$
\end{lem}
\begin{proof}	
	If $y\in W^c,$ then by Lemma \ref{lem:geom2}, $\abs{I(y,-,P)}\sim \frac{\ell(I)}{A}$ and hence $y\in W(s_y)$ for some $s_y\geq-\eta_Q.$ Clearly $\sup_{Q\in\calR_{\gamma}}\sup_{y\in W^c}s_y < \infty.$ Consequently, for a choice of $M$ large enough, $W^c\subset \cup_{s\in [\eta_Q,M)}W(s)$ and the point $(i)$ follows from the definition of $\varphi(\cdot,M).$
	
	By definition
	$$
	\abs{I(y,-,P)}\frac{A}{\ell(I)} \sim_{M} \varphi(y,M),
	$$ 
	hence $\varphi(\cdot,M)$ satisfies the point $(ii)$ with any choice of $M\geq 1.$ 
	
	Lastly, we check the point $(iii)$. Fix $z\in P$. If $\phi(z,+,W)\subset \cup_{s\in[\eta_Q,M )}W(s),$ then $\varphi(z+\gamma(t),M) = 1$ for all $t\in I(z,+,W)$ and the claim is clear. If $\phi(z,+,W)\not\subset  \cup_{s\in[\eta_Q,M )}W(s),$ then $\varphi(\phi(z,+,W),M) = \varphi(\phi\big(z,+,\big[\cup_{s\in[M,\infty)}W(s)\big]\big),M),$ i.e.  the function $\varphi(\cdot,M)$ already attains all possible values on the set $\phi\big(z,+,\big[\cup_{s\in[M,\infty)}W(s)\big]\big).$
	Then, as 
	\begin{align*}
			\varphi(z+\gamma(t),M) = \abs{I(z+\gamma(t),-,P)}\frac{A}{\ell(I)}2^{M},\qquad t\in I\big(z,+,\cup_{s\in[M,\infty)}W(s)\big),
	\end{align*}
	the claim follows from Lemma \ref{fact} below.
	%
\end{proof}
\begin{lem}\label{fact}
There holds that 
\begin{align}\label{q2}
	\lim_{A\to\infty}\sup_{\substack{ Q\in\calR_{\gamma} \\ z\in P \\  t,t'\in I(z,+,W) }}\frac{\abs{I(z+\gamma(t),-,P)}}{\abs{I(z+\gamma(t'),-,P)}} = 	
	1.
\end{align}
\end{lem}
\begin{proof}
  Fix $z\in P,$ denote $f_{z,t}(s) = z + s(1,-t)$ and note that 
	$$
2A\ell(I)\leq \abs{\gamma_2'(t)}\leq 2(A+N)\ell(I),\qquad t\in I(z,+,W)\subset I_A.
	$$
	Hence, for each $t\in I_A,$ there exists $t_h\in [2A\ell(I),2(A+N)\ell(I)]$ so that $
 	\abs{I(z+\gamma(t),-,P)} =  \abs{f_{z,t_h}^{-1}(P)}.$ 
 
First, assume that the lines $f_{z,t_h},f_{z,t'_h}$ exit the rectangle $P$ through the bottom and top edges. Then, there holds that $\abs{f_{z,t_h}^{-1}(P)}t_h = \ell(I)^2,$ and hence
 \begin{align}
 	 \frac{A}{A+N}\leq \frac{\abs{I(z+\gamma(t),-,P)}}{\abs{I(z+\gamma(t'),-,P)}} = \frac{t'_h}{t_h}\leq\frac{A+N}{A}
 \end{align}
from which the claim follows, with this configuration of the data.

Then, let  $t>t'$ and assume that the lines $f_{z,t_h},f_{z,t'_h}$ exit the rectangle $P$ from the right edge $\partial_r P$ (then, as $A$ is large, they exit $P$ through the top edge)  and respectfully let $e_t,e_{t'}\in\partial_r P$ be these points. By $t>t',$ it follows that $t_h>t'_h$ and $\pi_2(e_{t'})>\pi_2(e_t)$, and hence that
\begin{align}\label{fact0}
	\abs{f_{z,t'_h}^{-1}(P)} > \abs{f_{z,t_h}^{-1}(P)},\qquad \pi_2(v_{rt}-e_{t'})>\pi_2(v_{rt}-e_{t}).
\end{align}
Using the estimates on the line \eqref{fact0} and $\abs{f_{z,s_h}^{-1}(P)}s_h = \pi_2(v_{rt}-e_s),$ for $s\in\{t,t'\},$ we find that 
\begin{align}\label{xx}
	1> \frac{\abs{f_{z,t_h}^{-1}(P)}}{\abs{f_{z,t'_h}^{-1}(P)}} = \frac{\pi_2(v_{rt}-e_t)}{\pi_2(v_{rt}-e_{t'})}\frac{t'_h}{t_h} > \frac{t'_h}{t_h} > \frac{A}{A+N}.
\end{align}
From \eqref{xx} we find that
\begin{align*}
\frac{A}{A+N}	\leq \frac{\abs{f_{z,t_h}^{-1}(P)}}{\abs{f_{z,t'_h}^{-1}(P)}}  \leq \frac{A+N}{A}
\end{align*}
and the claim follows with this configuration of the data.

Lastly, we consider the case $t>t'$ when $f_{z,t_h}$ and $f_{z,t'_h},$ respectively, exit through the bottom edge and through the right edge. This case follows from the two cases above by writing 
\begin{align*}
	\frac{ \abs{f_{z,t_h}^{-1}(P)}}{ \abs{f_{z,t'_h}^{-1}(P)}} = 	\frac{ \abs{f_{z,t_h}^{-1}(P)}}{ \abs{f_{z,t''_h}^{-1}(P)}} \cdot 	\frac{ \abs{f_{z,t''_h}^{-1}(P)}}{ \abs{f_{z,t'_h}^{-1}(P)}}
\end{align*}  
for $t'_h<t''_h<t_h$ such that $f_{z,t''_h}$ passes through the vertex $v_{rb}$.
This last case along with the first two were representative of all possible cases, and as all the estimates were independent of the rectangle $Q,$ the proof is concluded.
\end{proof}

The other collection of functions we use is  $\{u_W\}_{Q\in\calR_{\gamma}},$ where  $u_W(y) = (g_W\circ\Xi)(y)$
and $\Xi$ is the following reflection
	\begin{align}\label{map:reflect}
	\Xi(x) = \big(w_d-(x_1-w_d)  ,x_2\big),\qquad w_d = \pi_1\big(\{x = (x_1,x_2)\in W: x_2 = \inf_{y\in W}y_2\}\big).
	\end{align}
The reflection $\Xi$ is exactly across the line depicted in Figure \ref{fig:all2} and the function $u_W$ is the symmetric version of $g_W$ with respect to this reflection.

\subsection{Approximate weak factorization}\label{sect:awf}
\subsubsection{The first two iterations}\label{sec:12} In this section we prove the following Proposition \ref{prop:pieceA}.
\begin{prop}\label{prop:pieceA} Let $f\in L^1_{\loc}$ be supported on a parabolic rectangle $Q=I\times J.$
Then, for all $A$ large enough (independently of $Q$), the function $f$ can be written as 
	\begin{align}\label{A1}
	f  = \left[h_QH_{\gamma}^*g_W - g_W H_{\gamma}h_Q\right] + \left[ h_W H_{\gamma}g_P - g_PH_{\gamma}^*h_W\right]  + \wt{f}_P,
	\end{align}  
where
	\begin{align}\label{A2}
	h_Q = \frac{f}{H_{\gamma}^*g_W},\qquad h_W = \frac{g_W H_{\gamma}h_Q}{H_{\gamma}g_P},\qquad \wt{f}_P = g_PH_{\gamma}^*\Big( \frac{g_W}{H_{\gamma}g_P}H_{\gamma}\big( \frac{f}{H_{\gamma}^*g_W}\big)\Big),
	\end{align} 
and there holds that
	\begin{align}\label{A3}
	\abs{h_Q} \lesssim_{A} \abs{f},\qquad 	\abs{h_W}\lesssim_{A} \Norm{f}{\infty} 1_W.
	\end{align}
	
Moreover, suppose that $\int_Q f = 0$ and let $\varepsilon>0.$ Then, for all $A$ large enough (independently of $Q$), there holds that
	\begin{align}\label{A4}
	\abs{\wt{f}_P}\lesssim  \varepsilon\Norm{f}{\infty}1_P.
	\end{align}
\end{prop}

 The identities on the lines \eqref{A1} and \eqref{A2} are simply algebraic, as long as the functions $h_Q, h_W$ are well-defined, which we will see below, and hence, it is enough to prove the estimates on the lines \eqref{A3} and \eqref{A4}. 

\begin{proof}[Proof of the estimates \eqref{A3}.]
We begin with a better estimate than is actually needed for \eqref{A3}, which will be reused in the proof of the estimate \eqref{A4}.
Let $\varepsilon>\varepsilon'>0$ and consider the set
	\begin{align*}
		I^c(x,+,W) = \{ t\in 	I(x,+,W) : g_W(x+\gamma(t)) = 1\}.
	\end{align*}
We choose  $A$ (independently of the data) so large that that
	\begin{align}\label{ik}
	\abs{	I^c(x,+,W)} \geq (1-\varepsilon')\abs{	I(x,+,W)}.
	\end{align}
Let us give a short argument for \eqref{ik}. Let $x\in Q$ be arbitrary and let $t_x$ be the smallest number with $y :=x+\gamma(t_x)\in\partial W.$ Let $c>0$ be a constant and define $r=c \frac{\ell(I)}{A}.$ Define the point $y':= x+\gamma(t_x+r)$ and let $R$ denote the rectangle with opposite vertices $y,y'.$ Then, by $\abs{\gamma_2'(t)}\sim A\ell(I),$ the rectangle $R$ has dimensions
\begin{align*}
	R = I\times J,\qquad \ell(I) = c \frac{\ell(I)}{A} ,\qquad  \ell(J) \sim c\ell(I)^2.
\end{align*}
Then, let $s_y$ be such that $y-\gamma(s_y) = v_{lb}$ and notice that the points $v_{lb}$ and $z' := y' - \gamma(s_y)$ are the opposite vertices of the rectangle $G = -\gamma(s_y) + R \subset P$ and that $\phi(y',-,P)\cap G =v_{rt}(G),$ where $v_{rt}(G)$ is the right-top vertex of $G.$ Since the rectangle $G$ has the same dimensions as $R,$ it follows with a choice of the constant $c$ large enough that $\phi(y'-,P)\cap P^c \not= \emptyset$ which implies that $g_W(y') = 1$ by property $(i)$ of the function $g_W.$ The same argument also works if we begin with "let $t_x$ be the \emph{largest} number with $y :=x+\gamma(t_x)\in\partial W.$" It follows that the two sections of the curve $\phi(x,+,W)$ where $g_W\not = 1$ both have lengths $\lesssim \frac{\ell(I)}{A}.$ By Lemma \ref{lem:geom1} we have $\abs{I(x,+,W)}\sim\ell(I)$ and hence we conclude that
\begin{align*}
	\abs{I(x,+,W)\setminus I^c(x,+,W)}\lesssim \frac{\ell(I)}{A},
\end{align*}
which implies \eqref{ik}.

Then, by $I(x,+,W)\subset I_A$ and \eqref{ik} we find
\begin{equation}\label{n1}
	\begin{split}
		\frac{(1-\varepsilon')\abs{	I(x,+,W)}}{(A+N)\ell(I)} \leq	\frac{	\abs{I^c(x,+,W)}}{(A+N)\ell(I)} &\leq	H_{\gamma}^*g_W(x) \\ 
		&= \int_{I(x,+,W)}g_W(x+\gamma(s))\frac{\ud s}{s} \leq  	\frac{	\abs{I(x,+,W)}}{A\ell(I)}.
	\end{split}
\end{equation}
It follows from Lemma \ref{lem:geom1} that with a choice of $A$ large enough for some absolute constants $\kappa_1,\kappa_2$ (independently of the data) there holds that 
\begin{align}\label{n22}
	0 < \kappa_1\leq \sup_{\substack{Q\in\calR_{\gamma} \\ x\in Q}}\frac{\abs{I(x,+,W)}}{\ell(I)}\leq \kappa_2 < \infty.
\end{align} 
Moreover, by Lemma \ref{lem:geom1}
\begin{align}\label{n222}
	\lim_{A\to\infty}\sup_{\substack{Q\in\calR_{\gamma} \\ x,x'\in Q}}\frac{\abs{I(x,+,W)}}{\abs{I(x',+,W)}} = 1.
\end{align}
Now choose arbitrary $x'\in Q$ and define $\kappa_{Q,A} = \frac{\abs{I(x',+,W)}}{\ell(I)}.$ Let $x\in Q$ be arbitrary, then, for a choice of $A$ large enough  combining both \eqref{n22} and \eqref{n222} it follows that  
\begin{align}\label{n2}
\kappa_{Q,A}(1-\varepsilon') \leq \frac{\abs{I(x,+,W)}}{\ell(I)} \leq \kappa_{Q,A}(1+\varepsilon'),\qquad 0<\kappa_1\leq \kappa_{Q,A} \leq \kappa_2 < \infty.
\end{align} 
Then, choose $A$ so large that \eqref{n1} and \eqref{n2} imply
	\begin{align}\label{n3}
			\frac{(1-\varepsilon)}{(A+N)}\kappa_{Q,A} \leq 	H_{\gamma}^*g_W(x) \leq 	\frac{(1+\varepsilon)}{A}\kappa_{Q,A}, \qquad 0<\kappa_1\leq \kappa_{Q,A} \leq \kappa_2 < \infty.
	\end{align}
By $I(y,-,P)\subset-2I_A$ we find
	\begin{align}\label{n4}
		\abs{	H_{\gamma}g_P(y)} &= \int_{I(y,-,P)}\frac{\ud s}{\abs{s}} \sim  \frac{\abs{I(y,-,P)}}{A\ell(I)}.
	\end{align}
It follows from \eqref{n3} and \eqref{n4}, respectively, that $h_Q,h_W$ are both well-defined and 
	\begin{align*}
		\abs{h_Q} \sim A\abs{f},\qquad \abs{h_W(y)}\sim \frac{A\ell(I)}{\abs{I(y,-,P)}}\abs{I(y,-,P)}\frac{A}{\ell(I)}	\abs{H_{\gamma}h_Q(y)} = A^2	\abs{H_{\gamma}h_Q(y)},
	\end{align*}
for the above estimate of $h_W,$ recall \eqref{q00}.
Estimating $\abs{H_{\gamma}h_Q(y)}$ a little further we find that
	\begin{align*}
		\abs{H_{\gamma}h_Q(y)} = \Babs{\int_{I(y,-,Q)}h_Q(y-\gamma(s))\frac{\ud s}{s}} &\lesssim \Norm{h_Q}{\infty}\int_{I(y,-,Q)}\frac{\ud s}{s} \\
		&\lesssim A\Norm{f}{\infty}\frac{\abs{I(y,-,Q)}}{A\ell(I)}\lesssim  A^{-1}\Norm{f}{\infty},
	\end{align*}
where in the last estimate we used Lemma \ref{lem:geom2}.
	Hence, we find  that $\abs{h_W} \lesssim  A  \Norm{f}{\infty}1_W\lesssim_A \Norm{f}{\infty}1_W.$
\end{proof}

\FloatBarrier
\begin{fig}[h]
	\centering
	\includegraphics[scale=0.175]{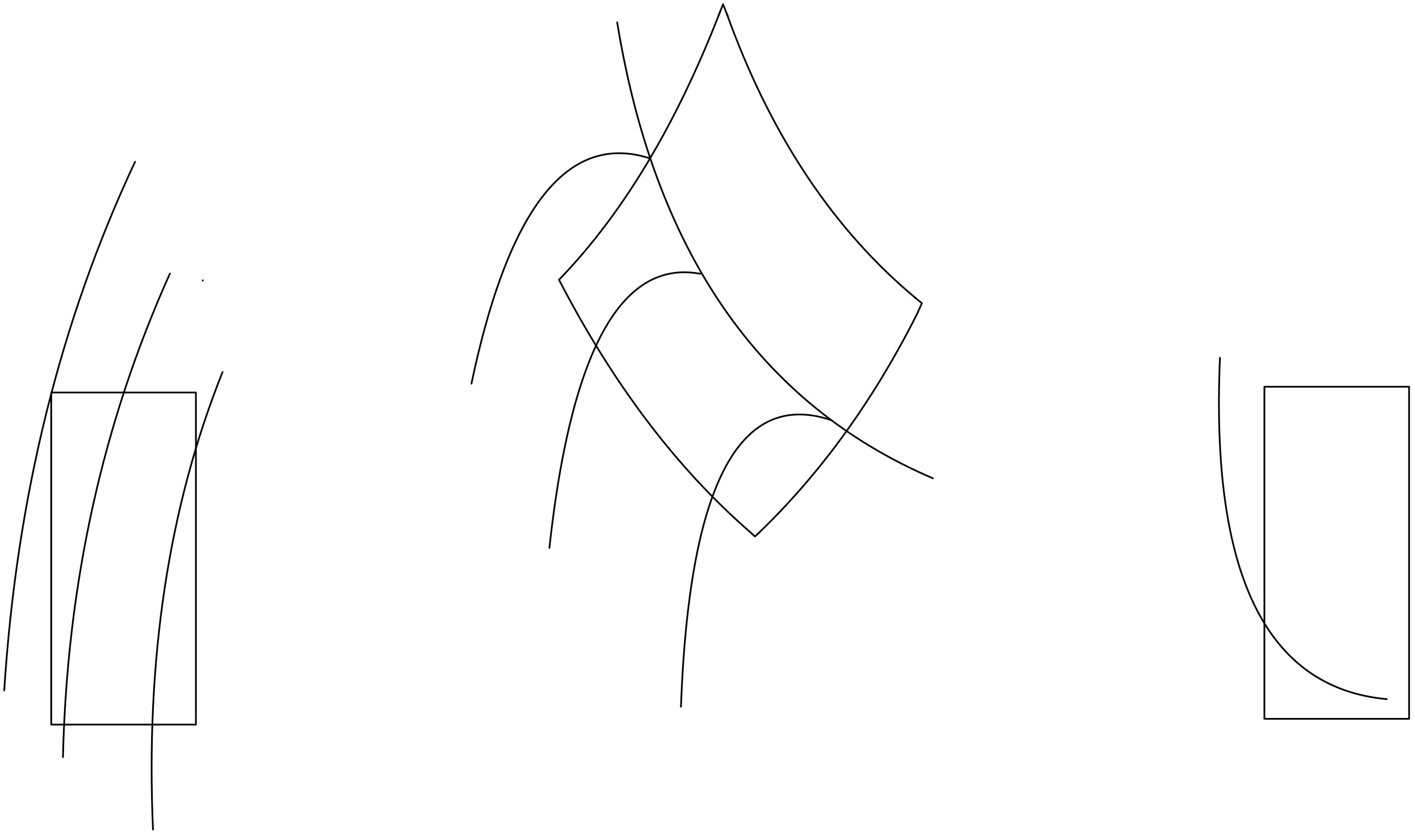}
	\caption{ Why \eqref{eq:crux2} holds; a parabola $\{z+\gamma(t): t \leq 0\}$ begins from a point $z\in P,$ exits $P,$ travels and penetrates into $W;$ then, another parabola $\{y-\gamma(t): t \geq 0\}$ begins from a point $y\in \phi(z,+,W),$ exits the set $W,$ travels and penetrates into the rectangle $Q.$
} 
	\label{fig2}
\end{fig}

	Now we come to the crucial part of the argument.
	Let $z\in P$ be arbitrary. Then, there holds that (see Figure \ref{fig2})
	\begin{align}\label{eq:crux2}
		Q &= \bigcup_{y\in\phi(z,+,W)}\phi(y,-,Q).
	\end{align}
There also holds that 
	\begin{align}\label{eq:crux3}
	\bigcup_{y\in\phi(z,+,W)}\phi(y,-,Q) = 
	\big\{z+\gamma(t)-\gamma(s): s\in I(z+\gamma(t),-,Q),\,t\in  I(z,+,W)\big\},
	\end{align}
easily checked from definitions.
Then, we notate
	\begin{align}\label{eq:Iz}
	I_zf = \int_{I(z,+,W)}\int_{I(z+\gamma(t),-,Q)} f(z+\gamma(t)-\gamma(s))\frac{\ud s}{s}\frac{\ud t}{t},
	\end{align}
and \eqref{eq:crux2}, \eqref{eq:crux3} show the double integral in \eqref{eq:Iz} to be over the rectangle $Q.$
Next, we recognize the density $\vartheta_z:Q\to \R_+$ satisfying
	\begin{align}\label{density}
		I_zf = \int_Q f\vartheta_z,\qquad \forall f\in L^1_{\loc}.
	\end{align}
Consider the mapping 
	\begin{align*}
		&h_z:F\to Q,\qquad F = \big\{(t,s):s\in I(z+\gamma(t),-,Q),\,t\in  I(z,+,W) \big\}, \\
		&h_z(t,s) = z +\gamma(t)-\gamma(s).
	\end{align*}
For a change of variables, we need to check that $h_z$ is bijective and differentiable. Differentiability is obvious and by \eqref{eq:crux2} and \eqref{eq:crux3} we have surjectivity. For injectivity, it is enough to show the following: let $y,y'\in\phi(z,+,W)$ be distinct, then $\phi(y,-,Q)\cap\phi(y',-,Q)=\emptyset.$ Notice that $\phi(y,-,Q)$ and $\phi(y',-,Q)$ are both contained in different translates of $\phi(0,-,\R_-\times \R_-),$ where $\R_- = (-\infty,0]$ (and $\R_+ = -\R_-$). The other fact we use is
\begin{align*}
\big(	a+\phi(0,-,\R_-\times \R_-)\big)&\cap\big( b + \phi(0,-,\R_-\times \R_-) \big)=\emptyset,\quad
a-b\in  \R_-\times \R_+\cup \R_+\times \R_-.
\end{align*}
Now injectivity follows by noting that $y-y' \in  \R_-\times \R_+\cup \R_+\times \R_-.$

Denote $\rho(t,s) =  (ts)^{-1}.$ Then, a change of variables tells us that 
	\begin{align*}
			I_zf = \int_F f\circ h_z(t,s)\frac{\ud (t,s)}{\rho(t,s)} = \int_Qf(x)\frac{\abs{\det J_{h_z^{-1}}(x)}}{\rho\circ h_z^{-1}(x)}\ud x,
	\end{align*}
and hence it remains to evaluate the density. There holds that 
	\begin{align*}
		\det J_{h_z^{-1}}(h_z(t,s)) = \frac{1}{\det J_{h_z}(t,s)},\qquad 
		\det J_{h_z}(t,s) = \det\begin{bmatrix}
		1 & -1 \\
		2t & -2s
		\end{bmatrix}= 2(t-s),
	\end{align*}
where for arbitrary $x\in Q$ we denote 
	\begin{align}\label{tox}
	x= 	h_z(t_x,s_x) = z+\gamma(t_x)-\gamma(s_x),\qquad-t_x\sim s_x\sim A\ell(I)
	\end{align} 
	for the unique choice of such $t_x,s_x.$
Then, we find that
	\begin{align}\label{density1}
		\vartheta_z(x) = \frac{\abs{\det J_{h_z^{-1}}(x)}}{\rho\circ h_z^{-1}(x)} = \frac{1}{2\abs{t_x-s_x}}\frac{1}{s_xt_x}.
	\end{align}

\begin{proof}[Proof of the estimate \eqref{A4}.] 
We begin with gathering three two sided estimates: \eqref{k1}, \eqref{k2} and \eqref{k4}.
Let $\varepsilon'>0.$ Let $x\in Q.$ Repeating the contents of the line \eqref{n3}, after a choice of $A$ sufficiently large, we find that
\begin{align}\label{k1}
\frac{(1-\varepsilon')}{(A+N)}\kappa_{Q,A} \leq 	H_{\gamma}^*g_W(x) \leq 	\frac{(1+\varepsilon')}{(A-1)}\kappa_{Q,A},\qquad 0<\kappa_1\leq \kappa_{Q,A}\leq \kappa_2 < \infty.
\end{align}

The property $(iii)$ (the line \eqref{q1}) shows that with a choice of $A$ large enough and an arbitrary $t''\in I(z,+,W),$  after defining $C_{z,Q,A} = g_W(z+\gamma(t'')),$ there holds that 
\begin{align}\label{q5}
\lim_{A\to\infty}\sup_{\substack{ Q\in\calR_{\gamma} \\ z\in P \\  t\in I(z,+,W) }}\frac{g_W(z+\gamma(t))}{C_{z,Q,A}} = 1.
\end{align}
By \eqref{q5} and the property $(ii)$ (the line \eqref{q00}), for $A$ sufficiently large, there holds that 
\begin{align}\label{k2}
(1-\varepsilon')C_{z,Q,A} \leq \abs{g_W(z+\gamma(t))}\leq  (1+\varepsilon')C_{z,Q,A},\qquad C_{z,Q,A} \sim  \abs{I(z+\gamma(t_z),-,P)}\frac{A}{\ell(I)},
\end{align}
where we fix some arbitrary choice of $t_z\in I(z,+,W).$

There holds that
\begin{align}\label{k3}
\frac{\abs{I(z+\gamma(t),-,P)}}{(A+N)\ell(I)}\leq \babs{H_{\gamma}g_P(z+\gamma(t))} = \int_{I(z+\gamma(t),-,P)}\frac{\ud s}{s} \leq \frac{\abs{I(z+\gamma(t),-,P)}}{(A-1)\ell(I)}.
\end{align}
By \eqref{q2} we find that with a choice of $A$ large enough \eqref{k3} implies
\begin{align}\label{k4}
	 \frac{(1-\varepsilon')\abs{I(z+\gamma(t_z),-,P)}}{(A+N)\ell(I)}\leq \babs{H_{\gamma}g_P(z+\gamma(t))} \leq \frac{(1+\varepsilon')\abs{I(z+\gamma(t_z),-,P)}}{(A-1)\ell(I)},
\end{align}
where both sides of the estimate now depend only on the fixed choice $t_z.$
Denote
	\begin{align*}
C_{z,A}(t,s) := \Big(\frac{g_W}{H_{\gamma}g_P}\Big)(z+\gamma(t))\Big(\frac{1}{H_{\gamma}^*g_W}\Big)(z+\gamma(t)-\gamma(s)).
	\end{align*}
Together \eqref{k1}, \eqref{k2} and \eqref{k4} show that 
\begin{equation}\label{j1}
	\begin{split}
		\abs{C_{z,A}(t,s)} &\leq (1+\varepsilon')C_{z,Q,A} \frac{(A+N)\ell(I)}{(1-\varepsilon')\abs{I(z+\gamma(t_z),-,P)}} \frac{A+N}{(1-\varepsilon')\kappa_{Q,A}} \\ &= \frac{1+\varepsilon'}{(1-\varepsilon')^2} \frac{C_{z,Q,A}\ell(I)}{A\abs{I(z+\gamma(t_z),-,P)}} \frac{A(A+N)^2}{\kappa_{Q,A}}.
	\end{split}
\end{equation}
Similarly, we find that
\begin{equation}\label{j2}
\begin{split}
	\abs{C_{z,A}(t,s)} &\geq (1-\varepsilon') C_{z,Q,A} \frac{(A-1)\ell(I)}{(1+\varepsilon')\abs{I(z+\gamma(t_z),-,P)}}\frac{A-1}{(1+\varepsilon')\kappa_{Q,A}} \\ 
&= \frac{1-\varepsilon'}{(1+\varepsilon')^2} \frac{C_{z,Q,A}\ell(I)}{A\abs{I(z+\gamma(t_z),-,P)}} \frac{A(A-1)^2}{\kappa_{Q,A}}.
\end{split}
\end{equation}
 Let $\varepsilon>0.$ As $ C_{z,Q,A} \sim  \abs{I(z+\gamma(t_z),-,P)}\frac{A}{\ell(I)}$ and $0<\kappa_1\leq \kappa_{Q,A}\leq \kappa_2 < \infty,$ we find with a choice of $A$ large enough from the estimates \eqref{j1} and \eqref{j2} that
\begin{align}\label{j3}
 (1-\varepsilon)\leq \frac{\abs{C_{z,A}(t,s)}}{A^3}C_{z,Q,A}^1 \leq (1+\varepsilon),
\end{align}
where $C_{z,Q,A}^1$ is some constant uniformly bounded from above and below (independently of the data $z,Q,A$). Denote $\wt{C}_z = C_{z,A}(t_z,s_z)$ for some fixed $(t_z,s_z)\in F$ so that especially \eqref{j3} is valid with $\wt{C}_z$ in place of $C_{z,A}(t,s).$ Consequently, 
\begin{align}\label{j4}
	\sup_{(t,s)\in F}\frac{\abs{C_{z,A}(t,s) - \wt{C}_z}}{A^3} \lesssim \varepsilon. 
\end{align}

Then,  we write out the error term to the extent that
	\begin{align*}
	\wt{f}_P(z) &= H_{\gamma}^*\Big( \frac{g_W}{H_{\gamma}g_P}H_{\gamma}\big( \frac{f}{H_{\gamma}^*g_W}\big)\Big)(z) \\
	&=\int_{I(z,+,W)}\Big(\frac{g_W}{H_{\gamma}g_P}\Big)(z+\gamma(t))H_{\gamma}\big( \frac{f}{H_{\gamma}^*g_W}\big)(z+\gamma(t))\frac{\ud t}{t} \\
	&= \int_{I(z,+,W)}\int_{I(z+\gamma(t),-,Q)}C_{z,A}(t,s)f(z+\gamma(t)-\gamma(s)) \frac{\ud s}{s}\frac{\ud t}{t} = \wt{C}_zI_zf +  I_{\Delta,z}f,
	\end{align*}
	where $I_zf$ was defined on the line \eqref{eq:Iz} and 
	\begin{align}
	I_{\Delta,z}f = \int_{I(z,+,W)}\int_{I(z+\gamma(t),-,Q)} (C_{z,A}(t,s) - \wt{C}_z)f(z+\gamma(t)-\gamma(s)) \frac{\ud s}{s}\frac{\ud t}{t}. 
	\end{align}	
	We apply the estimates
		\begin{align}\label{xxx}
	\abs{I(z,+,W)}\lesssim \ell(I),\qquad \abs{I(z+\gamma(t),-,Q)}\lesssim\frac{\ell(I)}{A}
	\end{align}
	and  \eqref{j4} to find
	\begin{align*}
		\abs{I_{\Delta,z}f} &\leq \Norm{f}{\infty}\int_{I(z,+,W)}\int_{I(z+\gamma(t),-,Q)}\abs{C_{z,A}(t,s) - \wt{C}_z} \frac{\ud s}{\abs{s}} \frac{\ud t}{\abs{t}} \\ 
			&\lesssim \Norm{f}{\infty}\sup_{(t,s)\in F} \frac{\abs{C_{z,A}(t,s) - \wt{C}_z}}{A^3} \lesssim \varepsilon\Norm{f}{\infty}.
	\end{align*}
	This estimate is of the desired form.
	Then, we analyse the term $I_zf.$ By the lines \eqref{density}, \eqref{density1} and using the zero-mean of the function $f,$ we write
	\begin{align*}
		I_zf =  \int_{I(z,+,W)}\int_{I(z+\gamma(t),-,Q)} f(h_z(t,s))\left( 1- \psi_z(h_z(s,t))\right)\frac{\ud s}{s}\frac{\ud t}{t},
	\end{align*}
	where 
	\[
	\psi_z(x) = \frac{\vartheta_z(c_Q)}{\vartheta_z(x)} = \frac{\abs{t_x-s_x}s_xt_x}{\abs{t_{c_Q}-s_{c_Q}}s_{c_Q}t_{c_Q}}.
	\]
	As
	\begin{align*}
			\abs{s_a-A\ell(I)},\abs{-t_a-A\ell(I)}\lesssim \ell(I),\qquad a\in \{x,c_Q\}
	\end{align*}
	we find that $\psi_z(x) \to 1,$ as $A\to\infty,$ and independently of the data $x,Q.$ Then, choosing $A$ so large that $\abs{1-\psi_z}\leq \varepsilon$ and again using the estimates on the line \eqref{xxx}, we find that 
	\begin{align*}
	\wt{C}_z\abs{I_zf} &\lesssim A^3 \int_{I(z,+,W)}\int_{I(z+\gamma(t),-,Q)} \abs{f(h_z(t,s))}\abs{1-\psi_z(h_z(s,t))}\frac{\ud s}{\abs{s}}\frac{\ud t}{\abs{t}} \\ 
	&\lesssim A^3\Norm{f}{\infty} \ell(I) \frac{\ell(I)}{A} \frac{\varepsilon}{A^2\ell(I)^2} = \varepsilon \Norm{f}{\infty}.
	\end{align*}
\end{proof}
%
%

\subsubsection{The last two iterations} Next, we repeat the contents of the previous Section \ref{sec:12}, but this time beginning from the rectangle $P$ instead of $Q.$
For the above arguments to pass through a second time we need to respect the symmetry present in the first iteration of the argument.  The only nonsymmetric object with respect to the reflection $\Xi$ in the statement of Proposition \ref{prop:pieceA} is $g_W.$ Hence, this time, we simply use the function $u_W = g_W\circ\Xi$ in place of $g_W.$ 

\begin{prop}\label{prop:pieceB} Let $f\in L^1_{\loc}$ be supported on a parabolic rectangle $P.$
	Then, for all $A$ large enough (independently of $P$), the function $f$ can be written as 
	\begin{align}\label{B1}
		f = \big[o_PH_{\gamma}^*u_W-u_WH_{\gamma}o_P\big] + \big[o_WH_{\gamma}g_Q - g_QH_{\gamma}^*o_W\big] + \wt{f}_Q,
	\end{align}
	where
	\begin{align}\label{B2}
		o_P = \frac{f}{H_{\gamma}^*u_W},\qquad o_W = \frac{u_WH_{\gamma}o_P}{H_{\gamma}g_Q},\qquad \wt{f}_Q = g_QH_{\gamma}^*\Big(\frac{u_W}{H_{\gamma}g_Q}H_{\gamma}\big(\frac{f}{H_{\gamma}^*u_W}\big)\Big),
	\end{align}
	and the following estimates hold
	\begin{align}\label{B3}
	\abs{o_P} \lesssim_{A} \abs{f},\qquad 	\abs{o_W}\lesssim_{A} \Norm{f}{\infty} 1_W.
	\end{align}
	
	Moreover, suppose that   $\int_P f = 0$ and let $\varepsilon>0.$   Then, for all $A$ large enough (independently of $P$), there holds that
		\begin{align}\label{B4}
		\abs{\wt{f}_Q}\lesssim  \varepsilon\Norm{f}{\infty}1_Q.
		\end{align}

\end{prop} 

\subsection{Closing the argument}\label{sect:close}
\begin{proof}[Proof of Theorem \ref{thm:main1}] Write
	\begin{align*}
	\int_Q\abs{b-\ave{b}_Q} = \int bf,\quad  	f = (\sigma-\ave{\sigma}_Q)1_Q,\quad \sigma = \sign(b-\ave{b}_Q),\quad \sign(\varphi) = \frac{\overline{\varphi}}{\abs{\varphi}}1_{\varphi\not=0}.
	\end{align*}
	According to the line \eqref{A1} of Proposition \ref{prop:pieceA} we factorize the function $f$ as
	\begin{align*}
		\int bf = \int b\left[h_QH_{\gamma}^*g_W - g_W H_{\gamma}h_Q\right] + \int b\left[ h_W H_{\gamma}g_P - g_PH_{\gamma}^*h_W\right]  + \int b\wt{f}_P.
	\end{align*}
	Then, by $\abs{f}\leq 2,$ and the estimates on the line \eqref{A3}, the first term above with brackets is controlled as
	\begin{equation}\label{h}
		\begin{split}
			\Babs{\int b\left[h_QH_{\gamma}^*g_W - g_W H_{\gamma}h_Q\right]} &= \Babs{\int g_W[b,H_{\gamma}]h_Q} \leq \Norm{[b,H_{\gamma}]}{L^p\to L^p}\Norm{g_W}{L^{p'}}\Norm{h_Q}{L^p} \\
		&\lesssim_{A}\Norm{[b,H_{\gamma}]}{L^p\to L^p}\abs{W}^{\frac{1}{p'}}\abs{Q}^{\frac{1}{p}} \sim \Norm{[b,H_{\gamma}]}{L^p\to L^p}\abs{Q},
		\end{split}
	\end{equation}
	where in the last estimate we used the estimates \eqref{size1}. The second term with brackets is similarly estimated to the same upper bound.
	Proceeding, accordingly to the line \eqref{B1} of Proposition \ref{prop:pieceB}, we factorize the function $\wt{f}_P$ as 
	\begin{align*}
		\int b\wt{f}_P = \int b\big[o_PH_{\gamma}^*u_W-u_WH_{\gamma}o_P\big] + \int b\big[o_WH_{\gamma}g_Q - g_QH_{\gamma}^*o_W\big] + \int b\wt{(\wt{f}_P)}_Q.
	\end{align*}
	 There holds that
	\begin{align}\label{mean:zero}
	\int_Q \wt{(\wt{f}_P)}_Q= \int_P \wt{f}_P = \int_Q f  = 0,
	\end{align}
	all of which are easy to check by moving the adjoints, see e.g. the similar argument for \eqref{eq:zm}.
	Then, by \eqref{A4} and \eqref{B4} there holds that $\abs{\wt{(\wt{f}_P)}_Q}\lesssim \varepsilon^2\lesssim 1.$
	Hence, the estimates on the lines \eqref{B3} and \eqref{B4} allow us to conclude, similarly as the estimate \eqref{h}, that
	\begin{align*}
		\Babs{\int b\big[o_PH_{\gamma}^*u_W-u_WH_{\gamma}o_P\big]} +  \Babs{\int  b\big[o_WH_{\gamma}g_Q - g_QH_{\gamma}^*o_W\big] }\lesssim_{A}  \Norm{[b,H_{\gamma}]}{L^p\to L^p}\abs{Q}.
	\end{align*}
	By \eqref{mean:zero} we find
	\begin{align*}
		 \Babs{\int b\wt{(\wt{f}_P)}_Q} = \Babs{\int_Q \big(b-\ave{b}_Q\big)\wt{(\wt{f}_P)}_Q}\leq \Norm{\wt{(\wt{f}_P)}_Q}{\infty}\int_Q\abs{b-\ave{b}_Q} \lesssim \varepsilon^{2}\int_Q\abs{b-\ave{b}_Q} .
	\end{align*}
	Putting all of the above together, we conclude that for some absolute constants $C_{A}, C$ (independent of $Q$) there holds that  
	\begin{align}\label{absorb}
		\int_Q\abs{b-\ave{b}_Q} \leq C_{A}\Norm{[b,H_{\gamma}]}{L^p\to L^p}\abs{Q} + C\varepsilon^2\int_Q\abs{b-\ave{b}_Q}.
	\end{align}
	As $b\in L^1_{\loc},$ the common term shared on both sides of the estimate \eqref{absorb} is finite. Hence, as $\varepsilon$ can be made arbitrarily small, by choosing $A$ sufficiently large, by absorbing the common term to the left-hand side  we find from \eqref{absorb} that
	\begin{align*}
		\int_Q\abs{b-\ave{b}_Q} \lesssim_A \Norm{[b,H_{\gamma}]}{L^p\to L^p}\abs{Q}.
	\end{align*}
	A division by $\abs{Q}$ closes the argument.
\end{proof}

\section{Extensions, open problems} In this section we present some extensions of Theorem \ref{thm:main1B} and Theorem \ref{thm:main2} along with some open problems. We are more interested in the parabolic case and hence only record the following Theorem \ref{thm:main2gen} as an extension to Theorem \ref{thm:main2}.

\begin{thm}\label{thm:main2gen} Let $b\in L^1_{\loc}(\R^d;\C)$ and $p\in(1,\infty).$ Then,
	\begin{align*}
	\Norm{b}{\bmo(\R^d)} \sim \sum_{i=1}^d\Norm{[b,H_{e_i}]}{L^p(\R^d)\to L^p(\R^d)}.
	\end{align*}
\end{thm}
It is clear how to adapt the proof of Theorem \ref{thm:main2} to prove Theorem \ref{thm:main2gen}.
Then, we move to discuss the immediately available extensions to Theorem \ref{thm:main1B}.

\subsection{Monomial curves}
A function $\gamma:\R\to\R^d$ is said to be a monomial curve if it is of the form
	\begin{align*}
		\gamma(t) = 
		\begin{cases}
		(\varepsilon_1\abs{t}^{\beta_1},\dots,\varepsilon_n\abs{t}^{\beta_d}),\, t > 0 \\ 
		(\delta_1\abs{t}^{\beta_1},\dots,\delta_n\abs{t}^{\beta_d}),\, t\leq 0, 
		\end{cases}
	\end{align*}
	where $\beta_i>0,$ $\varepsilon_i,\delta_i \in \{-1,1\},$ and there exists at least one index $j$ so that $\varepsilon_j\not=\delta_j.$ Let $\beta,\vare,\delta$ denote these parameter tuples.
Associated to a monomial curve $\gamma,$ and hence to the parameter tuple $\beta = (\beta_1,\dots,\beta_d),$ is the related bmo space; let $\calR_{\beta} = \calR_{\gamma}$ denote the collection of all rectangles $Q=I_1\times\dots\times I_d$ parallel to the coordinate axes such that $\ell(I_1)^{\frac{1}{\beta_1}}=	\ell(I_2)^{\frac{1}{\beta_2}} =\dots=\ell(I_d)^{\frac{1}{\beta_d}},$
and define the space $\BMO_{\beta}(\R^d)$ by the norm
\begin{align*}
	\Norm{b}{\BMO_{\beta}(\R^d)} = \sup_{Q\in\calR_{\beta}}\fint_Q\abs{b-\ave{b}_Q}.
\end{align*}
Notice that $\Norm{b}{{\BMO_{\beta}(\R^d)}}$ depends only on the vector $\beta$ and not on $\varepsilon,\delta\in\{-1,1\}^d.$  With this notation and $\gamma(t) = (t,t^2)$ we have $\BMO_{\gamma} = \BMO_{(1,2)}.$

For Theorem \ref{thm:main1B} the extension is the following. 
\begin{thm}\label{thm:main1gen}  Let $b\in L^1_{\loc}(\R^2;\C),$ let $p\in(1,\infty)$ and let $\gamma:\R\to\R^2$ be a monomial curve with the associated parameter tuples $\beta = (\beta_1,\beta_2),$ $\varepsilon = (\vare_1,\vare_2)$ and $\delta = (\delta_1,\delta_2).$ Let $\vare_i=\delta_i$ for exactly one index $i\in\{1,2\}.$  Then,
	\begin{align*}
		\Norm{b}{\BMO_{\beta}(\R^2)} \sim \Norm{[b,H_{\gamma}]}{L^p(\R^2)\to L^p(\R^2)}.
	\end{align*}
\end{thm}
\begin{proof}  The upper bound was proved in \cite{Bongers2019commutators} and holds for any $d\in\N$ and any monomial curve $\gamma:\R\to\R^d.$ Here we only need that $\vare_i\not=\delta_i$ for at least a single index $i\in \{1,2\}$. 
	
The lower bound follows by similar arguments as Theorem \ref{thm:main1} did. We need to make sure that a suitable geometry exists that allows for the approximate weak factorization (e.g. propositions \ref{prop:pieceA} and \ref{prop:pieceB}) which in turn allows the abstract awf argument (e.g. the proof of Theorem \ref{thm:main1} in Section \ref{sect:close}) to pass through.

In the parabolic case the geometry was formed by the three sets $Q,W(Q),P(Q)$ for $Q\in\calR_{\gamma},$ see Figure \ref{fig:all2}.  Next, we describe the correct geometry for any monomial curve $\gamma:\R\to\R^2,$ with exactly one index $\vare_i\not=\delta_i,$ and a rectangle $Q=I_1\times I_2\in\calR_{\beta}$. Set 
\begin{align*}
	\overline{\calJ}  &= \{i\in\{1,2\}: \varepsilon_i\not=\delta_i\},\qquad &\calJ&= \{1,2\}\setminus\overline{\calJ}, \\
	 \ell(Q) &= \ell(I_1)^{\frac{1}{\beta_1}} = \ell(I_2)^{\frac{1}{\beta_2}},\qquad &I_A&= [A\ell(Q),(A+N)\ell(Q)]
\end{align*}
and 
\begin{align*}
P(Q) &= P_1\times P_2,\qquad 	P_i = \begin{cases}
 \big((2A+N)\ell(Q)\big)^{\beta_i}	e_i + I_i ,\quad &i\in\overline{\calJ}, \\
	I_i,\quad &i\in\calJ
	\end{cases}
\end{align*}
and
\begin{align*}
\wt{Q} = \{x+\gamma(t): x\in Q,\quad t\in I_A\},\qquad 
\wt{P}(Q) &= \{z+\gamma(t): z\in P(Q),\quad t\in -I_A\}, \\
W(Q) &= \wt{Q}\cap\wt{P}(Q).
\end{align*}
With the above setup it is easy to verify the following key points of the argument
\begin{align*}
\abs{Q} \sim 	\abs{W(Q)}\sim \abs{P(Q)},\qquad Q = \bigcup_{y\in\phi(z,+,W(Q))}\phi(y,-,Q).
\end{align*}
The geometries of the  two cases are similar, with the case $\vare_1 = -\delta_1$ and $\vare_2= \delta_2$ being almost identical to the case of the parabola.

\end{proof}

When $d\geq 3$ the argument for the lower bound in Theorem \ref{thm:main1gen} does not pass through as such and a description of an essentially different geometry is needed. A further goal will be to relax the assumptions concerning the sign tuples $\vare,\delta$.
We leave this to a future work.

\begin{prbl}\label{prob1} Extend the lower bound in Theorem \ref{thm:main1gen} to the largest possible class of monomial curves and to the case $d\geq 3.$ 
\end{prbl}

\subsection{Off-diagonal}

Recall the following characterization of the homogeneous Hölder space
\begin{align}\label{alpha}
	\Norm{b}{\dot{C}^{\alpha,0}(\R^d)} = \sup_{x\not=y\in \R^d}\frac{\abs{b(x)-b(y)}}{\abs{x-y}^{\alpha}} \sim \sup_{Q} \ell(Q)^{-\alpha}\fint_{Q}\abs{b-\ave{b}_Q},   
\end{align}
where the supremum is taken over \emph{all} cubes $Q\subset \R^d$ and $\alpha>0.$ The reader unfamiliar with \eqref{alpha} can essentially read the proof from the proof of Proposition \ref{prop:alpha} below. Considering the right-hand side of \eqref{alpha} we set the following definition.

\begin{defn}\label{defn:paralpha} Let $b\in L^1_{\loc}(\R^d;\C)$, let $\alpha>0$ and let $\beta = (\beta_1,\dots,\beta_d)$ be a tuple of exponents (possibly associated to a monomial curve). Then, we define the norm
	\begin{align}\label{alpha0}
		\Norm{b}{\dot{C}^{\alpha,0}_{\beta}(\R^d)} =  \sup_{Q\in\calR_{\beta}} \abs{Q}^{-\alpha}\fint_{Q}\abs{b-\ave{b}_Q}.
	\end{align}
\end{defn}
Notice the difference in normalizations before the integrals on the lines \eqref{alpha} and \eqref{alpha0}.
Next, we will connect the norm of Definition \ref{defn:paralpha} with a pointwise definition.
\begin{defn}\label{defn:paralpha2} Let $\gamma:\R\to\R^2$ and for each $x\in\R^2$ let us denote 
	\begin{align*}
		\mathcal{X}_{\gamma}(x) =  \phi_{\gamma}(x,+,\R^2)\cup\phi_{\gamma}(x,-,\R^2)	,\qquad \phi_{\gamma}(a,\pm,B) = \{a\pm\gamma(t)\in B: t\in \R\}.
	\end{align*} 
Then, we define the norm
	\begin{align*}
		\Norm{b}{\dot{C}^{\alpha,0}_{\gamma}(\R^2)} = \sup_{x\in\R^2}\sup_{y\in	\mathcal{X}_{\gamma}(x) }\frac{\abs{b(x)-b(y)}}{\prod_{i=1}^{2}\abs{x_i-y_i}^{\alpha}}.
	\end{align*}
\end{defn}

\begin{prop}\label{prop:alpha} Let $\gamma:\R\to\R^2$ be a monomial curve with the associated parameter tuple $\beta = (\beta_1,\beta_2)$ and  $\vare_i\not=\delta_i$ for exactly one index $i\in\{1,2\}.$ Let $b\in L^1_{\loc}(\R^d;\C).$ Then, there holds that 
\begin{align*}
	\Norm{b}{\dot{C}^{\alpha,0}_{\gamma}(\R^2)} \sim 	\Norm{b}{\dot{C}^{\alpha,0}_{\beta}(\R^2)}.
\end{align*}
\end{prop}
\begin{proof} Fix a rectangle $Q=I\times J\in\calR_{\beta}$ and for distinct $x,y\in Q$ pick a point $z(x,y)\in \calX_{\gamma}(x)\cap\calX_{\gamma}(y)\cap 3Q.$ That such a point exists follows from $\varepsilon_i\not=\delta_i$ for exactly one index $i\in\{1,2\}.$ Then, we estimate
	\begin{align*}
		\fint_Q\abs{b-\ave{b}_Q} &\lesssim \fint_{Q}\fint_{Q}\abs{b(x)-b(z(x,y))}\ud x\ud y \\
		&\lesssim 	\Norm{b}{\dot{C}^{\alpha,0}_{\gamma}(\R^2)}  \fint_{Q}\fint_{Q}\prod_{i=1}^2\abs{x_i-z(x,y)_i}^{\alpha}\ud x\ud y \\ 
		&\lesssim \Norm{b}{\dot{C}^{\alpha,0}_{\gamma}(\R^2)} \ell(I)^{\alpha}\ell(J)^{\alpha} = \Norm{b}{\dot{C}^{\alpha,0}_{\gamma}(\R^2)} \abs{Q}^{\alpha}.
	\end{align*}
	
Let $x\in\R^2$ and $y\in\calX_{\gamma}(x).$ Let $Q=I\times J\in\calR_{\beta}$ be the rectangle with the points $x,y$ as vertices. We let $Q_k(x) = I_k(x)\times J_k(x)\in\calR_{\beta}$ be the rectangle such that $x\in Q_k(x)\subset Q$ and $\ell(I_k(x))= 2^{-k}I.$ Similarly for the point $y.$ Then, we write
\begin{align*}
	b(x)-b(y) = \sum_{k=0}^{\infty} \big( \ave{b}_{Q_{k+1}(x)}-\ave{b}_{Q_{k}(x)}\big) -  \sum_{k=0}^{\infty}\big(\ave{b}_{Q_{k+1}(y)}-\ave{b}_{Q_{k}(y)}  \big),
\end{align*} 
where we use the Lebesgue differentiation theorem with $\calR_{\beta}$ rectangles. A comment on this. For the standard argument to pass through, we need to know that for each point $x\in \R^2$ there exists a sequence $Q_k(x)\to x$ with $\diam(Q_k(x))\to 0$ (obviously true) and that the related maximal function $M_{\calR_{\beta}}f(x) = \sup_{x\in Q\in\calR_{\beta}}\fint_Q\abs{f}$ is of weak type $(1,1).$ To see the weak type $(1,1)$, dominate $M_{\calR_{\beta}}\leq\sum_{i=1}^N M_{\calD_{\beta}^i}$ by a finite number of dyadic operators over some anisotropic dyadic grids $\calD_{\beta}^i.$ Such grids are constructed at least in  \cite{ClaOu2017}, and these grids have all the important properties that we would expect of a dyadic grid, hence, a standard argument shows that $M_{\calD_{\beta}^i}$ is of weak type $(1,1).$
With this detail in the clear, we estimate the second of the martingale differences as
\begin{align*}
	\sum_{k=0}^{\infty}\babs{\ave{b}_{Q_{k+1}(y)}-\ave{b}_{Q_{k}(y)}} &\leq \sum_{k=0}^{\infty}\fint_{Q_{k+1}(y)}\abs{b-\ave{b}_{Q_k(y)}} \lesssim \sum_{k=0}^{\infty}\Norm{b}{\dot{C}^{\alpha,0}_{\beta}(\R^2)}\abs{Q_k(y)}^{\alpha} \\ 
	&\lesssim \Norm{b}{\dot{C}^{\alpha,0}_{\beta}(\R^2)}\abs{Q}^{\alpha}=  \Norm{b}{\dot{C}^{\alpha,0}_{\beta}(\R^2)}\prod_{i=1}^2\abs{x_i-y_i}^{\alpha}.
\end{align*}
The first martingale difference estimates identically and we conclude.
\end{proof}

\begin{thm}\label{thmex:alpha} Let $b\in L^1_{\loc},$ let $1<p<q<\infty$ and define $\alpha = \frac{1}{p}-\frac{1}{q}.$  Let $\gamma:\R\to\R^2$ be a monomial curve with the associated parameter tuples $\beta, \vare, \delta$ and $\vare_i\not=\delta_i$ for exactly one index $i\in\{1,2\}.$ Then, there holds that 
	\begin{align*}
		\Norm{b}{\dot{C}^{\alpha,0}_{\beta}(\R^2)} \lesssim \Norm{[b,H_{\gamma}]}{L^p(\R^2)\to L^q(\R^2)}.
	\end{align*}
\end{thm}
\begin{proof} Follows by the same proof as Theorem \ref{thm:main1} in Section \ref{sect:close} with the only difference being the replacement of $\Norm{[b,H_{\gamma}]}{L^p\to L^p}$ with $\Norm{[b,H_{\gamma}]}{L^p\to L^q}$ which gives the correct normalization.  
\end{proof}

\begin{rem}\label{rem1} If Problem \ref{prob1} has a positive solution by the awf argument, then this automatically leads to the generalization of Theorem \ref{thmex:alpha} to $d\geq 3$ and any monomial curve as therein.
\end{rem}

 One is tempted to attempt to prove the upper bound $\Norm{[b,H_{\gamma}]}{L^p(\R^2)\to L^q(\R^2)}\lesssim \Norm{b}{\dot{C}^{\alpha,0}_{\beta}(\R^2)}$ as follows. 
	By Proposition \ref{prop:alpha} and
	\begin{align*}
	\abs{ [b,H_{\gamma}]f(x)} &= \Babs{ p.v.\int (b(x)-b(x-\gamma(t)))f(x-\gamma(t))\frac{\ud t}{t}} \\ 
	&\qquad\qquad\leq \Norm{b}{\dot{C}^{\alpha,0}_{\gamma}(\R^2)}\int \abs{f(x-\gamma(t))}\prod_{i=1}^2\abs{\gamma_i(t)}^{\alpha}\frac{\ud t}{\abs{t}},
	\end{align*}
	it would be enough to bound the following fractional integral
	\begin{align*}
	I_{\gamma}^{\alpha}f(x) = \int f(x-\gamma(t))\prod_{i=1}^2\abs{\gamma_i(t)}^{\alpha}\frac{\ud t}{\abs{t}} = \int f(x-\gamma(t))\frac{\ud t}{\abs{t}^{1-\alpha\abs{\beta}}},
	\end{align*}
	where $\abs{\beta}=\sum_{i=1}^2\abs{\beta_i}.$
	It is not clear whether this operator is bounded $L^p\to L^q$ or not. In the other direction, a standard scaling argument shows that if $I^{\alpha}_{\gamma}$ is bounded $L^p\to L^q$ then necessarily $\alpha = \frac{1}{p}-\frac{1}{q}.$
		
	For completeness, we give this scaling argument in $\R^d.$
	For $\vec{\lambda} \in\R^d$ and $x\in\R^d$ denote $\vec{\lambda}x = (\lambda_1 x_1,\dots,\lambda_d x_d)$ and $\Dil_{\vec{\lambda}}f(x) = f(\vec{\lambda}x).$  For $\beta\in\R^d$ and $\lambda\in\R$ denote $\lambda^{\beta} = (\lambda^{\beta_1},\dots,\lambda^{\beta_d}).$ Dilation structure for $I_{\gamma}^{\alpha}$ is contained in the following identity,
	\begin{align*}
	\Dil_{\lambda^{\beta}}I^{\alpha}_{\gamma}f(x) &= \int f(\lambda^{\beta}x-\gamma(t))\abs{t}^{\alpha\abs{\beta}}\frac{\ud t}{t} \\ 
	&= \int \Dil_{\lambda^{\beta}}f(x-\gamma(t)) \lambda^{\alpha\abs{\beta}} \abs{t}^{\alpha\abs{\beta}}\frac{\ud t}{t} = \lambda^{\alpha\abs{\beta}}I^{\alpha}_{\gamma}\Dil_{\lambda^{\beta}}f(x).
	\end{align*}
	Then, assuming $\Norm{I^{\alpha}_{\gamma}}{L^p(\R^d)\to L^q(\R^d)}< \infty,$ and by $\Norm{\Dil_{\vec{\lambda}}f}{L^s(\R^d)} = \prod_{i=1}^d\lambda_i^{-\frac{1}{s}}\Norm{f}{L^{s}(\R^d)},$ we find that 
	\begin{align*}
	\lambda^{-\frac{\abs{\beta}}{q}}\Norm{I^{\alpha}_{\gamma}f}{L^q(\R^d)}  = \Norm{	\Dil_{\lambda^{\beta}}I^{\alpha}_{\gamma}f }{L^q(\R^d)} &= \lambda^{\alpha\abs{\beta}}\Norm{ I^{\alpha}_{\gamma}\Dil_{\lambda^{\beta}}f}{L^q(\R^d)} \\ 
	& \lesssim \lambda^{\alpha\abs{\beta}}\Norm{\Dil_{\lambda^{\beta}}f}{L^p(\R^d)} = \lambda^{\alpha\abs{\beta}} \lambda^{-\frac{\abs{\beta}}{p}}\Norm{f}{L^p(\R^d)}.
	\end{align*}
	Choosing $f$ so that both sides of the above estimate are positive and varying $\lambda$ shows that necessarily $\alpha = \frac{1}{p}-\frac{1}{q}.$

	With the above discussion at hand we are led to two problems.
\begin{prbl}  Let $1<p<q<\infty,$ $\alpha = \frac{1}{p}-\frac{1}{q}$ and let $\gamma$ be a monomial curve. Does there hold that $\Norm{[b,H_{\gamma}]}{L^p\to L^q}\lesssim \Norm{b}{\dot{C}^{\alpha,0}_{\beta}}?$ 
\end{prbl}

\begin{prbl} Let $1<p<q<\infty$ and $\alpha = \frac{1}{p}-\frac{1}{q}.$ For what curves $\gamma,$ is $I^{\alpha}_{\gamma}$ bounded $L^p\to L^q?$
\end{prbl}

%
%
%

The second off-diagonal case is when $q<p.$ 
\begin{thm}\label{thm:super} Let $b\in L^1_{\loc}(\R^2;\C)$ and let $\gamma:\R\to\R^2$ be a monomial curve such that $\vare_i=\delta_i$ for exactly one index $i\in\{1,2\}$. Let $1<q<p<\infty$ and define the exponent $r>1$ by the relation  $\frac{1}{q} = \frac{1}{r} + \frac{1}{p}.$ Then, there holds that 
	\begin{align*}
		\Norm{b}{\dot{L}^r(\R^2)}\sim \Norm{[b,H_{\gamma}]}{L^{p}(\R^2)\to L^q(\R^2)}.
	\end{align*}
\end{thm}
\begin{proof} It was shown in Stein, Wainger \cite{SteWai1978} that $H_{\gamma}:L^p\to L^p,$ $1<p<\infty$ is a bounded operator.
	 By the boundedness of $H_{\gamma}$, that the commutator is unchanged modulo constants and Hölder's inequality we find that
	\begin{align*}
		\Norm{[b,H_{\gamma}]}{L^p\to L^q} &= 	\Norm{[b-c,H_{\gamma}]}{L^p\to L^q}  \leq \Norm{(b-c)H_{\gamma}}{L^p\to L^q}  + \Norm{H_{\gamma}(b-c)}{L^p\to L^q} \\ 
		&\leq \Norm{b-c}{L^r}\Norm{H_{\gamma}}{L^p\to L^p} + \Norm{H_{\gamma}}{L^q\to L^q}\Norm{f\mapsto (b-c)f}{L^q\to L^q} \\
		&\leq   \Norm{b-c}{L^r}\Norm{H_{\gamma}}{L^p\to L^p} +  \Norm{H_{\gamma}}{L^q\to L^q}\Norm{b-c}{L^r} \\ 
		&\lesssim \Norm{b-c}{L^r},
	\end{align*}
	for any constant $c\in\C.$  This upper bound is valid for any $d$ and any monomial curve.
	
For the lower bound we only discuss the case of the parabola and it is clear how to adapt the proof to any monomial curve $\gamma:\R\to\R^2$ with the sign tuples agreeing in exactly one entry.	
With the factorization results at hand, the lower bound follows with the same argument as the case $q<p$ in \cite{HYT2021JMPA}. The important points to consider when adapting the proof are the following:
\begin{enumerate}[$(i)$]
\item If $K\subset\R^d$ is compact, then $K\subset Q$ for some $Q\in\calR_{\gamma}.$
\item Any parabolic rectangle $Q\in\calR_{\gamma}$ allows \emph{parabolic} stopping time arguments with constants independent of $Q.$
\item Each $Q\in\calR_{\gamma}$ allows a factorization of functions, as in propositions \ref{prop:pieceA} and \ref{prop:pieceB}.
\item There exists $\lambda > 0$ so that for each $Q\in\calR_{\gamma},$ there holds that 
\begin{align*}
\abs{S}\sim \abs{Q},\qquad  Q\subset \lambda S,\qquad S\in\{W(Q),P(Q)\},
\end{align*}
where the implicit constants do not depend on $Q$ and the dilation $\lambda S$ of any connected bounded set $S$ is
\begin{align*}
	\lambda S = \big\{y: y_i =c_{\pi_i(S) } +\lambda(x_i-c_{\pi_i(S) }),\quad x\in S,\quad  i = 1,\dots,d \big\}.
\end{align*}
Since $S$ is connected and bounded $\pi_i(S)$ is a finite interval and the meaning of the centre point $c_{\pi_i(S)}$ is clear.
\end{enumerate}
We leave the details to the interested reader.  
\end{proof}

\begin{rem} Theorem \ref{thm:super} extends to the case $d\geq 3$ as long as we can carry out the awf argument in that setting, similarly as with Remark \ref{rem1}.
\end{rem}

The author has no competing interests to declare.

\bibliography{references}

\end{document}